\documentclass[11pt,a4paper,oneside, reqno]{amsart}
\usepackage{mathrsfs}
\usepackage{amsfonts}
\usepackage{amssymb}
\usepackage{amsxtra}
\usepackage{dsfont}
\usepackage{amsthm}
\usepackage{geometry}
\usepackage{float}
\usepackage{amsmath, amssymb, amsthm}
\usepackage[utf8]{inputenc}
\usepackage[T1]{fontenc}
\usepackage{graphicx}
\usepackage[unicode, psdextra]{hyperref}
\usepackage{listings}
\usepackage{indentfirst}
\usepackage{theoremref}
\usepackage{thmtools}
\usepackage{physics}
\usepackage{mathtools}
\usepackage[polish,english]{babel}
\usepackage{color}
\usepackage{bbm}

\newcommand{\pl}[1]{\foreignlanguage{polish}{#1}}

\usepackage{hyphenat}

\newcommand{\vj}{j}
\newcommand{\mR}{\mathcal{R}}
\newcommand{\mK}{\mathcal{K}}
\newcommand{\la}{\lambda}
\newcommand{\R}{\mathbb{R}}
\newcommand{\CC}{\mathbb{C}}
\newcommand{\N}{\mathbb{N}}

\newcommand{\Q}{\mathbb{Q}}
\newcommand{\mH}{\mathcal{H}}
\newcommand{\mD}{\mathcal{D}}
\newcommand{\mF}{\mathcal{F}}
\newcommand{\mP}{\mathcal{P}}
\newcommand{\mS}{\mathcal{S}}
\newcommand{\mM}{\mathcal{M}}

\newcommand{\tR}{\widetilde{R}}

\newcommand{\tK}{\widetilde{K}}

\newcommand{\tg}{\widetilde{\gamma}}

\definecolor{green}{rgb}{0,0.5,0}

\newcommand{\ind}[1]{\mathbbm{1}_{#1}}

\newtheorem{theorem}{Theorem}[section]
\newtheorem{lemma}[theorem]{Lemma}
\newtheorem{pro}[theorem]{Proposition}
\newtheorem{cor}[theorem]{Corollary}
\theoremstyle{remark}

\theoremstyle{definition}

\numberwithin{equation}{section}

\title[Dimension-free $L^p$ estimates for higher order maximal Riesz transforms]{Dimension-free $L^p$ estimates for higher order maximal Riesz transforms in terms of the Riesz transforms}

\author{Maciej Kucharski}
\address{Maciej Kucharski\\
	Instytut Matematyczny\\
	Uniwersytet \pl{Wroc{\l}awski}\\
	Plac Grun\-waldzki 2\\
	50-384 \pl{Wroc{\l}aw}\\
	Poland}
\email{maciej.kucharski@math.uni.wroc.pl}

\author{B{\l}a{\.z}ej Wr{\'o}bel}
\address{B{\l}a{\.z}ej Wr{\'o}bel\\
		Instytut Matematyczny\\
			Polska Akademia Nauk \\
			\'Sniadeckich 8\\
			00–656 Warszawa\\
	Poland \&
	Instytut Matematyczny\\
	Uniwersytet \pl{Wroc{\l}awski}\\
	Plac Grun\-waldzki 2\\
	50-384 \pl{Wroc{\l}aw}\\
	Poland}
\email{blazej.wrobel@math.uni.wroc.pl}

\author{Jacek Zienkiewicz}
\address{Jacek Zienkiewicz\\
	Instytut Matematyczny\\
	Uniwersytet \pl{Wroc{\l}awski}\\
	Plac Grun\-waldzki 2\\
	50-384 \pl{Wroc{\l}aw}\\
	Poland}
\email{jacek.zienkiewicz@math.uni.wroc.pl}


\thanks{Maciej Kucharski and B{\l}a{\.z}ej Wr{\'o}bel were supported by the National Science Centre (NCN), Poland  research project Preludium Bis 2019/35/O/ST1/00083. B{\l}a{\.z}ej Wr{\'o}bel was also supported by the National Science Centre (NCN), Poland  research project Sonata Bis 2022/46/E/ST1/00036.}

\begin{document}

	\begin{abstract}
	We prove a dimension-free $L^p(\R^d)$, $1<p<\infty$, estimate for the vector of higher order maximal Riesz transforms in terms of the corresponding Riesz transforms. This implies a dimension-free $L^p(\R^d)$ estimate for the vector of maximal Riesz transforms in terms of the input function. We also give explicit estimates for the dependencies of the constants on $p$ when the order is fixed.  Analogous dimension-free estimates are also obtained for single higher order Riesz transforms with an improved estimate of the constants. 
	\end{abstract}
	
	\maketitle

	\section{Introduction}
	\label{sec:int}
	Fix a positive integer $k$ and denote by $\mH_k = \mH_k^d$ the space of spherical harmonics of degree $k$ on the Euclidean sphere $S^{d-1}.$  Throughout the paper we identify $P\in \mH_k$ with the corresponding solid spherical harmonic. Via this identification $P\in \mH_k$ is a harmonic polynomial on $\R^d$ which is homogeneous of degree $k,$ i.e.\ satisfies $P(x)=|x|^k P(x/|x|),$ $x\in \R^d.$    

For $P \in \mH_k$ the Riesz transform $R=R_P$ is defined by the kernel
\begin{equation}
	\label{eq:KP}
	K_P(x)=K(x) = \gamma_k \frac{P(x)}{\abs{x}^{d+k}} \qquad\textrm{ with } \qquad  \gamma_k = \frac{\Gamma\left( \frac{k+d}{2}\right)}{\pi^{d/2}\Gamma\left( \frac{k}{2}\right)},
\end{equation}
more precisely,
\begin{equation} \label{eq:R}
	R_P f(x) = \lim_{t \to 0^+} R_P^t f(x), \qquad\textrm{ where } \qquad R_P^t f(x) = \gamma_k \int_{\abs{y}>t} \frac{P(y)}{\abs{y}^{d+k}} f(x-y) \, dy.
\end{equation}
The operator $R_P^t$ is called the truncated Riesz transform.  In the particular case of $k=1$ and $P_j(x)=x_j$ the operators $R_{P_j},$ $j=1,\ldots,d,$ coincide with the classical first order Riesz transforms. It is well known, see \cite[p. 73]{stein}, that the Fourier multiplier associated with the Riesz transform $R_P$ equals 
\begin{equation} \label{eq:m}
	m_P(\xi) = (-i)^k P(\xi/|\xi|),\qquad \xi\in \R^d.
\end{equation}
By the above formula $m_P$ is bounded and 
Plancherel's theorem implies the $L^2(\R^d)$ boundedness of $R_P.$
The $L^p(\R^d)$ boundedness of the single Riesz transforms $R_P$ for $1<p<\infty$ follows from the Calder\'on--Zygmund method of rotations \cite{CZ2}.

The systematic study of the dimension-free $L^p$ bounds
for the Riesz transforms has begun in the seminal paper of
E. M. Stein \cite{stein_riesz}. There he proved a dimension-free $\ell^2$ vector-valued estimate
for the vector of the first order Riesz transforms
\begin{equation}
	\label{eq:Riesz0}
	\norm{\bigg(\sum_{j=1}^d |R_jf|^2\bigg)^{1/2}}_{L^p(\R^d)}\leqslant C_p\, \|f\|_{L^p(\R^d)},\qquad 1<p<\infty.
\end{equation}
In the inequality above $R_j,$ $j=1,\ldots,d,$ denote the first order Riesz transforms defined via \eqref{eq:R} with $P_j(x)=x_j$ and the constant $C_p$ is independent of the dimension $d.$ 

Stein's result has been extended to many other settings. The analogue of the dimension-free inequality \eqref{eq:Riesz0} has been also proved for higher order Riesz transforms, see \cite[Th\'{e}or\`{e}me 2]{duo_rubio}. The optimal constant $C_p$ in \eqref{eq:Riesz0} remains unknown when $d\geqslant 2;$ however  the best results to date  given in \cite{BW1} (see also \cite{DV})  established the correct order of the dependence on $p$. We note that the explicit values of $L^p(\R^d)$ norms of the single first order Riesz transforms $R_j,$ $j=1,\ldots,d,$ were obtained by Iwaniec and Martin \cite{iwaniec_martin} based on the method of rotations.

In this paper we study  the relation between $R_P$ and the maximal Riesz transform defined by
\begin{equation*}
	R_P^* f(x) = \sup_{t > 0} \abs{R_P^t f(x)}.
\end{equation*}
Clearly, we have the pointwise inequality $R_P f(x)\leqslant R_P^*f(x).$ 
In a series of papers \cite[Theorem 1]{mateu_verdera} (first order Riesz transforms), \cite[Section 4]{mopv} (odd order higher Riesz transforms), and \cite[Section 2]{mov1} (even order higher Riesz transforms), J. Mateu, J. Orobitg, C. P\'erez, and J. Verdera proved that also a reverse inequality holds in the $L^p(\R^d)$ norm. 
Namely, together the results of  \cite{mateu_verdera,mopv,mov1} imply that for each $1<p<\infty$ there exists a constant $C(p,k,d)$ such that
\begin{equation} \label{eq:mat_ver}
	\norm{R_P^* f}_{L^p(\R^d)}\leqslant C(p,k,d)\norm{R_P f}_{L^p(\R^d)}
\end{equation}
for all $f\in L^p(\R^d)$. As a matter of fact, the estimate \eqref{eq:mat_ver} has been proved in \cite{mateu_verdera,mopv,mov1}  for more general singular integral operators with
even kernels \cite{mov1} or with odd kernels \cite{mopv}. However, even for the higher order Riesz transforms, the values of $C(p,k,d)$ that follow from these papers grow exponentially with the dimension. In view of \cite{Jan}, the question about an improved rate arises naturally.

The first step towards a dimension-free estimate of the constant $C(p,k,d)$ in \eqref{eq:mat_ver} has been made by the first and the second author, who proved that when $p=2$, in  \eqref{eq:mat_ver} one may take an explicit dimension-free constant $C(2,1,d)\leqslant 2\cdot 10^8,$ see \cite[Theorem 1.1]{kw}. The arguments applied in \cite{kw} relied on Fourier transform estimates together with 
square function techniques developed by Bourgain \cite{bou1}, and Bourgain, Mirek, Stein, and Wr\'obel \cite{BMSW1,bourgain_wrobel}, for studying dimension-free estimates for maximal functions associated with symmetric convex bodies.

Recently Liu, Melentijević, and Zhu \cite{LiuMeZhu} extended the results of \cite{kw} and proved that $C(p,1,d)\leqslant (2+1/\sqrt{2})^{2/p},$ for $p\ge 2.$ An important ingredient of their argument is the positivity of the transition kernels $M_1^t,$ which is not at all clear in \cite{kw}.


In this paper we prove that the dimension-free estimate of the form \eqref{eq:mat_ver} and its vector-valued generalization hold for Riesz transforms of arbitrary order $k$ and for all $1<p<\infty$. The main result of our paper is the following square function estimate of the vector of maximal Riesz transforms in terms of the Riesz transforms.

%
%




\begin{theorem} \thlabel{thm1}
	Take $p \in (1, \infty)$ and let $k\leqslant d$ be a non-negative integer. Let $\mP_k$ be a subset of $\mH_k.$  Then there is a constant $A(p,k)$ independent of the dimension $d$ and such that
	\begin{equation*} 
		\norm{\left(\sum_{P\in \mP_k} |R_P^* f|^2\right)^{1/2}}_{L^p(\R^d)} \leqslant A(p,k) \norm{ \left(\sum_{P\in \mP_k} |R_P f|^2\right)^{1/2}}_{L^p(\R^d)},
	\end{equation*}
	where $f\in L^p(\R^d).$ Moreover, for fixed $k$ we have
	\[
		A(p,k) = O(p^{5/2+k/2}) \quad \textrm{as } \ p\to \infty \qquad \textrm{and} \qquad A(p,k) = O((p-1)^{-5/2-k/2}) \quad \textrm{as } \ p\to 1.
	\]
\end{theorem}

In particular, if $\mP_k$ contains one element $P,$ then \thref{thm1} immediately gives
\begin{equation*}
	\norm{ R_P^* f}_{L^p(\R^d)} \leqslant A(p,k) \norm{  R_P f}_{L^p(\R^d)}.
\end{equation*}
In this case however, we can slightly improve the constant $A(p,k).$
\begin{theorem} \thlabel{thm2}
	Take $p \in (1, \infty)$ and let $k\leqslant d$ be a non-negative integer. Let $P$ be a spherical harmonic of degree $k.$ Then there is a constant $ B(p,k)$ independent of the dimension $d$ and such that
	\begin{equation*} 
		\norm{ R_P^* f}_{L^p(\R^d)} \leqslant B(p,k) \norm{  R_P f}_{L^p(\R^d)},
	\end{equation*}
	where $f\in L^p(\R^d).$ Moreover, for fixed $k$ we have
	\[
		B(p,k) = O(p^{2+k/2}) \quad \textrm{as } \ p\to \infty \qquad \textrm{and} \qquad B(p,k) = O((p-1)^{-2-k/2}) \quad \textrm{as } \ p\to 1.
	\]
\end{theorem}


Our last main result follows from a combination of \thref{thm1} with  a result of Duoandikoetxea and Rubio de Francia \cite[Th\'eor\`eme 2]{duo_rubio}. Denote by $a(d,k)$ the dimension of $\mH_k$ and let $\{Y_j\}_{j=1,\ldots,a(d,k)}$ be an orthogonal basis of $\mH_k$ normalized by the condition
\[
\frac{1}{\sigma(S^{d-1})}\int_{S^{d-1}} |Y_j(\theta)|^2\,d\sigma(\theta) =\frac{1}{a(d,k)};
\]
here $d\sigma$ denotes the (unnormalized) spherical measure. 
\begin{cor}
	\thlabel{cor:dr}
	Take $p \in (1, \infty)$ and let $k\leqslant d$ be a non-negative integer. Then there is a constant $G(p,k)$ independent of the dimension $d$ and such that
	\begin{equation*} 
		\norm{\left(\sum_{j=1}^{a(d,k)} |R_{Y_j}^* f|^2\right)^{1/2}}_{L^p(\R^d)} \leqslant G(p,k) \|f\|_{L^p(\R^d)},
	\end{equation*}
	where $f\in L^p(\R^d).$ Moreover, for fixed and odd $k$ we have
	\[
		G(p,k) = O(p^{7/2+k/2}) \quad \textrm{as } \ p\to \infty \qquad \textrm{and} \qquad G(p,k) = O((p-1)^{-7/2-k}) \quad \textrm{as } \ p\to 1
	\] 
	and for even $k$ we have
	\[
		G(p,k) = O(p^{9/2+k/2}) \quad \textrm{as } \ p\to \infty \qquad \textrm{and} \qquad G(p,k) = O((p-1)^{-9/2-k}) \quad \textrm{as } \ p\to 1.
	\] 
	
\end{cor}

\subsection{Structure of the paper and our methods}
There are four main ingredients used in the proofs of  \thref{thm1,thm2}.

Firstly, we need a factorization of the truncated Riesz transform $R_P^t=M^t_k(R_P)$. Here, $M^t_k,$ $t>0,$ is a family of radial Fourier multiplier operators. In the case $k=1$ this factorization has been one of the key steps in establishing the main results of \cite{kw}. In particular the operator $M_1^t$ considered here coincides with $M^t$ defined in \cite[(eq.) 3.5]{kw}. For general values of $k$ the factorization is  implicit in \cite[Section 2]{mateu_verdera} ($k=1$), \cite[Section 2]{mov1} ($k$ even), and \cite[Section 4]{mopv} ($k$ odd). Note that for the first order Riesz transforms the formulas $R_j^t=M^t_1(R_j),$ $j=1,\ldots,d,$ together with the identity $I=-\sum_{j=1}^d R_j^2$ imply that 
\begin{equation}
	\label{eq:MtR1} 
	M^t_1=-\sum_{j=1}^d M_1^t R_j^2=-\sum_{j=1}^d R_j^t R_j.\end{equation}
Details of the factorization procedure are given in Section \ref{sec:fa}.

The second ingredient we need is an averaging procedure. It turns out that a useful analogue of \eqref{eq:MtR1}  is not directly available for Riesz transforms of orders higher than one. The reason behind it is the fact that not all compositions of first-order Riesz transforms are higher order Riesz transforms according to our definition. For instance, in the case $k=3$  the multiplier symbol of $R_1^3=R_1 R_1 R_1$ on $L^2(\R^2)$ equals $-i\xi_1^3/|\xi|^3$  and $P(\xi)=-i\xi_1^3$ is not a spherical harmonic. However, the formula $$I=-\sum_{j_1=1}^d \sum_{j_2=1}^d \sum_{j_3=1}^d R_{j_1}^2 R_{j_2}^2 R_{j_3}^2,$$
includes squares of all compositions of Riesz transforms including $R_1^6=(R_1^3)^2$. Therefore the above formula does not directly lead to an expression of $M^t_k$ in terms of $R^t_P$ and $R_P.$ To overcome this problem we average over the special orthogonal group $SO(d).$  Then we obtain
\begin{equation} 
	\label{eq:lemAp}
	M^t_kf(x) = C(d,k)\int_{SO(d)}   \sum_{j \in I} (R_j^t R_j f)_U(x)\, d\mu(U),
\end{equation}
see \thref{pro:av}. Here $T_U$ is the conjugation of an operator $T$ by $U\in SO(d),$ see \eqref{conj_U_def}, $d\mu$ denotes the normalized Haar measure on $SO(d),$  while $C(d,k)$ is a constant. The symbol $I$ denotes the set of  multi-indices $j=(j_1,\ldots,j_k)$ with pairwise distinct elements  while $R_j^t$ and $R_j$ are the truncated Riesz transforms and the Riesz transforms \eqref{eq:R} corresponding to the monomials $P_j(x)=x_{j_1}\cdots x_{j_d}.$ Note that since $j\in I$ the polynomials $P_j$ are spherical harmonics and thus the operators $R_j$ are indeed higher order Riesz transforms. In view of \eqref{eq:lemAp},  if we demonstrate that $C(d,k)$ is bounded by a universal constant, we are left with estimating the maximal function corresponding to
$\sum_{j \in I} R_j^t R_j.$ The reduction via the averaging procedure is described in detail in Section \ref{sec:av}. It is noteworthy that in order for the averaging approach to work it is essential that for each order $k$ the multiplier symbols of $M^t_k$ are radial functions.

The third main ingredient of our argument is an extension to $\CC^d$ followed by the complex method of rotations of Iwaniec and Martin \cite{iwaniec_martin}. We use the complex method of rotations to estimate the maximal function $\tR^*$ corresponding to 
\begin{equation}
	\label{eq:tRtintro}
	\tR^t:=\sum_{j\in I} \tR_j^t \tR_j.
\end{equation} Here $\tR_j^t$ and $\tR_j$ denote extensions to $\CC^d$ of the truncated Riesz transform $R_j^t$ and the Riesz transform $R_j.$ The definition of $\tR_j$ can be given on the multiplier level according to the scheme from \cite{iwaniec_martin}. We note, however, that the truncated extended operator $\tR_j^t$ needs to be defined differently --- on a kernel level. In the context of dimension-free estimates for Riesz transforms the real method of rotations has been employed by Duoandikoetxea and Rubio de Francia \cite{duo_rubio}. However, as it can be applied only to operators with odd kernels, for the general case we need the complex version. The method of rotations itself is preceded by a number of other ingredients. In particular we need $L^p$ vector-valued estimates for the maximal directional truncated $k$-th power of the complex Hilbert transform, see \thref{pro:hzt}, and for the vector of higher order Riesz transforms, see \thref{pro:tRve}. En route to obtain these results we also need Khintchine's inequalities and specific computations. All of it reflects the size of the constants $A(p,k)$ in \thref{thm1} and $B(p,k)$ in \thref{thm2}. The extension procedure and the application of the complex method of rotations are described in detail in Section \ref{sec:cmr}.

The last ingredient is a restriction procedure. This allows us to deduce the estimates for $R^*$ on $\R^d$ from the estimates for $\tR^*$  on $\CC^d$. The restriction of the complex Riesz transforms $\tR_j$ in \eqref{eq:tRtintro} can be done on the multiplier level as in \cite[Chapter 4]{iwaniec_martin}. However, in order to restrict $\tR_j^t$ and the maximal function $\tR^*$ we need to work on the kernel level. A problem that we encounter here is that the resulting restricted operator of $\tR^*$ is not the same as the desired maximal operator $R^*$. Therefore we need to investigate their difference and estimate it appropriately. The restriction procedure is described in Section \ref{sec:re}.

At the first reading it might be helpful to skip the explicit values of constants in terms of $k$ and $p$ and only focus on these constants being independent of the dimension $d.$ An interested reader may trace the exact dependencies of the constants in terms of $k$ and $p$ in the paper.

\subsection{Notation}
\label{sec:not}
We finish the introduction with a description of the notation and conventions used in the rest of the paper. 

\begin{enumerate}

	\item The letters $d$ and $k$ stand for the dimension and for the order of the Riesz transforms, respectively.  In particular we always have $k\leqslant d,$ even if this is not stated explicitly.  
	
	\item The symbol $\N$ represents the set of positive integers. Throughout the paper we assume that $k\in \N$. We write $\mathbb{Q}_+$ for the set of positive rational numbers. 
	
	\item By $[d]$ we denote the set $\{1,\ldots,d\}$ of positive integers up to $d.$

	\item For an exponent $p\in [1,\infty]$ we let $q$ be its conjugate exponent satisfying
	$$1=\frac1p+\frac1q.$$
	When $p\in(1,\infty)$ we set
	$$p^*:=\max(p,(p-1)^{-1}).$$

	\item We abbreviate $L^p(\R^d)$ to $L^p$ and $\norm{\cdot}_{L^p}$ to $\norm{\cdot}_p$. For a sublinear operator  $T$ on $L^p$ we denote by $\|T\|_{p\to p}$ its norm. We let $\mathcal S(\R^d)=\mathcal S$ be the space of Schwartz functions on $\R^d.$ Slightly abusing the notation we say that a sublinear operator $T$  is bounded on $L^p$ if it is bounded on $\mathcal S$ in the $L^p$ norm.  For $k\in \N$ we let $\mD(k)$ be the linear span of  $\{R_P(f)\colon P\in \mH_k, f\in \mS\}.$ Since $R_P$ is bounded on $L^p$ for $1<p<\infty$   the space $\mD(k)$ is then a subspace of each of the $L^p$ spaces.
	
	\item For a Banach space $E$ the symbol $L^p(\R^d;E)$ stands for the space of weakly measurable functions  $f\colon \R^d\to E$ with the norm $\|f\|_{L^p(\R^d;E)}=(\int_{\R^d}\|f(x)\|_E^p\, dx)^{1/p}.$ Similarly, for a finite set $F$ by $\ell^p(F;E)$ we denote the Banach space of $E$-valued sequences $\{f_s\}_{s\in F}$  with the norm $\|f\|_{\ell^p(F;E)}=(\sum_{s\in F}\|f_s\|_E^p)^{1/p}.$
	
	\item The symbol $C_{\Delta}$ stands for a constant that possibly depends on $\Delta>0.$ We write $C$ without a subscript when the constant is universal in the sense that it may depend only on $k$ but not on the dimension $d$ nor on any other quantity.
	
	\item For two quantities $X$ and $Y$ we write $X\lesssim_{\Delta} Y$  if $X \leqslant C_{\Delta} Y$ for some constant $C_{\Delta}>0$ that depends only on $\Delta.$ We abbreviate $X\lesssim Y$ when $C$ is a universal constant. We also write $X\sim Y$ if both $X\lesssim Y$ and $Y \lesssim X$ hold simultaneously. By $X\lesssim^{\Delta} Y$ we mean that $X\leqslant C^{\Delta} Y$ with  a universal constant $C.$ Note that in this case $X^{1/\Delta}\lesssim Y^{1/\Delta}.$  
	
	\item The symbol $S^{d-1}$ stands for the $(d-1)$-dimensional unit sphere in $\R^d$ and by $\omega$ we denote the uniform measure on $S^{d-1}$ normalized by the condition $\omega(S^{d-1})=1.$ We also write 
	\begin{equation}
		\label{eq:Sd-1}
		S_{d-1} = \frac{2\pi^{d/2}}{\Gamma\left( \frac{d}{2} \right)}
	\end{equation} to denote the unnormalized surface area of $S^{d-1}.$   
We write $\zeta$ for the uniform measure on $S^{2d-1}$ normalized by the condition $\zeta(S^{2d-1})=1.$
\item We let 
\begin{equation}
	\label{eq:gk}
	\gamma_k =\gamma_{k,d}:= \frac{\Gamma\left( \frac{k+d}{2}\right)}{\pi^{d/2}\Gamma\left( \frac{k}{2}\right)}\qquad{{\rm and}}
\qquad \tg_k=\gamma_{k,2d}= \frac{\Gamma\left( d+\frac{k}{2}\right)}{\pi^{d}\Gamma\left( \frac{k}{2}\right)}
\end{equation}

	\item The Fourier transform is defined for $f\in L^1$ and $\xi\in\R^d$  by the formula
\begin{equation*}
	\mF{f}(\xi) = \hat{f}(\xi)=\int_{\R^d} f(x) e^{-2 \pi i x \cdot \xi} \, dx.
	\end{equation*}

	\item The Gamma function is defined for $s>0$ by the formula
	\[
	\Gamma(s) = \int_0^\infty t^{s-1}e^{-t} \, dt.
	\]
	We shall use Stirling's approximation for $\Gamma(s)$ 
	\begin{equation}
		\label{StirF}
		\Gamma(s)\sim\sqrt{2\pi}s^{s-\frac12}e^{-s},\qquad s\to \infty.
	\end{equation}
A useful consequence of \eqref{StirF} is the formula
\begin{equation}
	\label{StirFra}
	\Gamma(s+\alpha)\sim s^{\alpha}\Gamma(s),\qquad s\to \infty
\end{equation}
which is valid for each fixed $\alpha\ge 0.$

\item We will also need the following formula
\begin{equation}
	\label{eq:Bform}
		2\int_0^{\infty} \frac{r^{d-1}}{(1+r^2)^{d+\alpha}}\,dr=B\left(\frac{d}{2},\frac{d}{2}+\alpha\right)=\frac{\Gamma(\frac{d}{2})\Gamma(\frac{d}{2}+\alpha)}{\Gamma(d+\alpha)},
\end{equation}
valid for $\alpha\ge 0.$ This follows from change of variables $r^2\to r$ followed by  formulas for Euler's Beta function $B(a,b)$ from \cite[5.12.1, 5.12.3]{nist}.
\end{enumerate}

\section{Factorization}
\label{sec:fa}
The goal of this section is to show that a factorization formula for $R_P^t$ in terms of $R_P$ is feasible. Proposition below is implicit in \cite[Section 4]{mopv} and  \cite[pp.\ 1435--1436]{mov1}.
\begin{pro}
	\thlabel{pro:fact}
	Let $k\in \N$. Then there exists a family of  operators $M_k^t,$ $t>0$, 
	which are bounded on $L^p,$ $1<p<\infty,$ and such that for all $P\in \mH_k$ we have 
	\begin{equation}
		\label{eq:fact}
		R_P^t f=M_k^t (R_P f),
	\end{equation}
	where $f\in L^p.$ Each $M^t_k$ is a convolution operator with a radial convolution kernel $b^t_k.$ Moreover, when $P\in \mH_k$ and $f\in \mS,$ then for a.e.\ $x\in \R^d,$ the function $t\mapsto M_k^t (R_P f)(x)$ is continuous on $(0,\infty).$   
\end{pro}
\begin{proof}
We consider separately the cases of $k$ odd or even starting with $k$ odd.   

Let $c_d=\frac{\Gamma((d-1)/2)}{2\pi^{d/2}\Gamma(1/2)},$ $N=(k-1)/2,$ and denote by $B$ the open Euclidean ball of radius $1$ in $\R^d.$ 
It is justified in \cite[pp.\ 3674--3675]{mopv} that the function
\begin{equation}
	\label{eq:bdef}
	b(x)=b_{k,d}(x):=\sum_{j=1}^d R_j\left[y_j\cdot h(y)\right](x),
\end{equation}
where
$$h(y)=c_d(1-d)\frac{1}{|y|^{d+1}}\ind{B^c}(y)+(\beta_1 +\beta_2|y|^2+\cdots+\beta_N  |y|^{2N-2})\ind{B}(y),$$
satisfies the formula 
\begin{equation}
	\label{eq:RKB}
	R_P(b)(x)=K_P(x)\ind{B^c}.
\end{equation}
Here $\beta_1,\ldots,\beta_N$ are constants which depend only on $k$ and $d$ and whose exact value is irrelevant for our considerations, and $K_P,$ $R_P$ have been defined in \eqref{eq:KP}, \eqref{eq:R}, respectively. The important point is that \eqref{eq:RKB} remains true for any $P\in \mH_k.$ 

Denote by $H$ the radial profile of the Fourier transform of $h$, i.e.\ $H(|\xi|)=\widehat{h}(\xi)$ for $\xi\in \R^d.$ By taking the Fourier transform of \eqref{eq:bdef} it is straightforward to see that $b$ is a radial function. This follows since the multiplier symbol of $R_j$ is $-i\xi_j/|\xi|$ and  $$\widehat {(y_j h(y))}(\xi)=\frac{\xi_j}{-2\pi i|\xi|}\,H'(|\xi|),$$ so that
\begin{align*}
	\mF b(\xi)=\sum_{j=1}^d \frac{\xi_j^2}{2\pi |\xi|^2}\cdot H'(|\xi|)=\frac{1}{2\pi}H'(|\xi|)
\end{align*}
is indeed radial and so is $b.$

Let $b^t(x)=b^t_k(x):=t^{-d}b(x/t)$ be the $L^1$ dilation of $b;$ clearly $b^t$ is still radial.  The dilation invariance of $R_P$ together with \eqref{eq:RKB} leads us to the expression
\begin{equation}
	\label{eq:RKBt}
	K_P(x)\ind{B^c}(x/t)=R_P(b^t)(x).
\end{equation}
Let $M^t_k$ be the convolution operator  
\begin{equation*}
	M^t_k f(x)=b^t*f(x).
\end{equation*}
It follows from \cite[Section 4]{mopv} that $M_k^t$ is bounded on $L^p$ spaces whenever $1<p<\infty.$ 
Moreover, in view of \eqref{eq:RKBt} we see that
\begin{equation*}
	R_P^t f= R_P(b^t)*f=b^t * R_P(f)=M_k^t (R_P f).
\end{equation*}

Finally, for $f\in \mS,$ $P\in \mH_k,$  and $x\in \R^d$ the mapping $t\mapsto R_P^t f(x)$ is  continuous on $(0,\infty).$ Thus, also $M_k^t (R_P f)(x)$ is a continuous function of $t>0$ for a.e.\ $x.$ This completes the proof of the proposition in the case when $k$ is odd.

It remains to consider $k$ even. 
Denote $N=k/2.$ 
	From (10) and (12) in \cite[pp.\ 1435--1436]{mov1} it follows that the function
	\begin{equation*}
		b(x)=b_{k,d}(x):=(\alpha_0+\alpha_1|x|^2+\cdots+{\alpha_{N-1}}|x|^{2(N-1)})\ind{B}(x)
	\end{equation*}
	satisfies the formula 
	\begin{equation}
		\label{eq:RKBev}
		R_P(b)(x)=K_P(x)\ind{B^c}(x).
	\end{equation}
	Here $\alpha_1,\ldots,\alpha_{N-1}$ are constants which depend only on $k$ and $d$ and whose exact value is irrelevant for our considerations. As in the case of odd $k,$ the important point is that \eqref{eq:RKBev} remains true for any $P\in \mH_k.$ 
	
Using \eqref{eq:RKBev} we proceed as in the proof in the case when $k$. Let $b^t(x)=b^t_k(x):=t^{-d}b(x/t)$ be the $L^1$ dilation of $b.$ Since $b$ is clearly radial the same is true of  $b^t$.  Let $M^t_k$ be the convolution operator  
	\begin{equation*}
		M^t_k f(x)=b^t*f(x).
	\end{equation*}
	It follows from \cite[Section 2]{mov1} that $M_k^t$ is bounded on $L^p$ spaces whenever $1<p<\infty.$ 
	Moreover, in view of \eqref{eq:RKBev} we see that
	\begin{equation*}
		R_P^t f= R_P(b^t)*f=b^t * R_P(f)=M_k^t (R_P f).
	\end{equation*}
Moreover, for $f\in \mS,$ $P\in \mH_k,$  and $x\in \R^d$ the mapping $t\mapsto R_P^t f(x)$ is  continuous on $(0,\infty)$ and therefore so is $t\mapsto M_k^t (R_P f)(x)$. This completes the proof of the proposition.

%
%
\end{proof}

As a corollary of  \thref{pro:fact} we see that in order to justify \thref{thm1,thm2} it suffices to control vector and scalar-valued maximal functions corresponding to the operators $M^t_k.$ Note that by \thref{pro:fact} 
for $f\in \mD(k)$ we have
\begin{equation*}
	\sup_{t>0}|M^t_k f(x)|=\sup_{t\in \Q_+ }|M^t_k f(x)|.
\end{equation*}
In particular $\sup_{t>0}|M_k^t f(x)|$ 
 is measurable for such $f$, although possibly being infinite for some $x.$ 
 Define
 \begin{equation}
 	\label{eq:MtmaxQ}
 	M^*f(x)=\sup_{t\in \Q_+}|M^t_k f(x)|.
 \end{equation}
 \thref{pro:fact} reduces our task to proving the following two theorems.
\begin{theorem} \thlabel{thm1'}
	Fix $k\in \N.$ For each $p\in (1,\infty)$  there is a constant $A(p,k)$ independent of the dimension $d$ and such that for any $S\in\N$  we have
	\begin{equation*} 
		\norm{\left(\sum_{s=1}^S |M^* f_s|^2\right)^{1/2}}_{p} \leqslant A(p,k) \norm{ \left(\sum_{s=1}^S | f_s|^2\right)^{1/2}}_{p},
	\end{equation*}
	where $f_1,\ldots,f_S \in  L^p.$ Furthermore $A(p,k)$ satisfies $A(p,k) \lesssim_k (p^*)^{5/2+k/2}.$ 
\end{theorem}
\begin{theorem} \thlabel{thm2'}
	Fix $k\in \N.$ For each $p\in (1,\infty)$  there is a constant $ B(p,k)$ independent of the dimension $d$ and such that
	\begin{equation*} 
		\norm{ M^* f}_{p} \leqslant B(p,k) \norm{  f}_{p},
	\end{equation*}
	whenever $f\in L^p.$ Moreover $B(p,k)$ satisfies $B(p,k) \lesssim_k (p^*)^{2+k/2}.$
\end{theorem}

\section{Averaging}
\label{sec:av}
In this section we describe the averaging procedure. The averaging procedure will allow us to pass from $M^*$ to another maximal operator that is better suited for applications in Sections \ref{sec:cmr} and \ref{sec:re}. Before moving on, we establish some notation. For a multi-index $$j = (j_1, \dots, j_k) \in \{1, \dots, d\}^k\quad \text{we write} \quad P_j(x) =x_j:= x_{j_1} \cdots x_{j_k}$$ and denote by $R_j$ the Riesz transform $R_{P_j}$ associated with the monomial $P_j.$ The truncated transform $R_j^t$ and the maximal transform $R_j^*$ are defined analogously. We also abbreviate $K_{j}(x)=K_{P_j}(x)$ and $K_j^t(x)=K_{P_j}^t(x).$   
As we will be mainly interested in multi-indices with different components, we define
\[
	I = \{j \in \{1, \dots, d\}^k: j_k \neq j_l \text{ for } k \neq l \}.
\] 

The averaging procedure will provide an expression  for $M^t_k$ in terms of the Riesz transforms $R_j$ and $R_j^t$ postulated in \eqref{eq:lemAp}. For $f\in L^p,$ $1<p<\infty,$ denote
\begin{equation*}
	R^tf:=\sum_{j \in I} R_j^t R_jf\qquad\textrm{and let}\qquad 	R^* f:=\sup_{t\in \Q_+}\left|R^t f\right|.
\end{equation*}
Note that both $R^t$ and $R^*$ are well defined on all $L^p$ spaces. Indeed, $R_j^t$ and $R_j$ are bounded on $L^p$ and the supremum in the definition of $R^*$ runs over a countable set thus defining a measurable function.

Let $SO(d)$ be the special orthogonal group in dimension $d.$ Since it is compact, it has a bi-invariant Haar measure $\mu$ such that $\mu(SO(d))=1.$ For $U\in SO(d)$ and a sublinear operator $T$ on $L^2$ we denote by $T_U$ the conjugation by $U,$ i.e. the operator acting via 	\begin{equation}\label{conj_U_def}
	T_Uf(x)=T(f (U^{-1}\cdot))(Ux).
\end{equation}

\begin{pro} \thlabel{pro:av}
	Fix $k\in \N.$ Then there is a constant $C(d,k)\in \R$ such that 
	\begin{equation} 
		\label{eq:lemAplem}
		M^t_k f(x)=C(d,k)\int_{SO(d)} [(R^t)_U f](x) d\mu(U)
	\end{equation}
	for all $t>0$ and $f\in L^p.$ Moreover, $|C(d,k)|$ has an estimate from above by a constant that depends only on $k$ but not on the dimension $d,$ so that 
	\begin{equation} \label{eq:lemA}
		\left(\sum_{s=1}^S \abs{M^* f_s(x)}^2\right)^{1/2} \lesssim \int_{SO(d)}  \left(\sum_{s=1}^S \abs{[(R^*)_U f_s](x)}^2\right)^{1/2} \, d\mu(U),
	\end{equation}
	for $S\in \N$ and $f_1,\ldots,f_S\in L^p.$ 
\end{pro}
\begin{proof}
	Let  $A$ be the operator
	\begin{equation}
		\label{eq:Adef}
		A = \sum_{j \in I} R_j^2 ,
	\end{equation}
	which by \eqref{eq:m} means that its multiplier symbol equals
	\[
	a(\xi) =(-i)^{2k} \sum_{j \in I} \frac{\xi_j^2}{\abs{\xi}^{2k}}= (-1)^k \sum_{j \in I} \frac{\xi_j^2}{\abs{\xi}^{2k}}.
	\]
	Let  $\widetilde{A}$ be the operator with the multiplier symbol
	\begin{equation}
		\label{eq:mtil}
		\widetilde{a}(\xi) := \int_{SO(d)} a(U\xi) \, d\mu(U)= (-1)^k \sum_{j \in I} \int_{SO(d)} \frac{\left( (U\xi \right)_j)^2}{\abs{\xi}^{2k}} \, d\mu(U).
	\end{equation}
	Then
	$\widetilde{a}$ being
	radial and homogeneous of order $0$ is constant. 
	
	The first step in the proof of the proposition is to show that
	\begin{equation}
		\label{eq:lemA1}
		|\widetilde{a}|\sim 1
	\end{equation}
	uniformly in the dimension $d.$ Note that each of the integrals on the right hand side of \eqref{eq:mtil} has the same value independently of $j\in I,$
	so that
	\[
	\widetilde{a}(\xi)= (-1)^k \abs{I} \int_{SO(d)} \frac{(\left(U\xi \right)_{(1,\ldots,k)})^2}{\abs{\xi}^{2k}} \, d\mu(U);
	\]
	here $|I|$ stands for the number of elements in $I.$ Since $\tilde{a}$ is radial, integrating the above expression over the unit sphere $S^{d-1}$ with respect to the normalized surface measure $d\omega$ we obtain
	\begin{equation} \label{eq0}
		\widetilde{a} = (-1)^k \abs{I} \int_{S^{d-1}} \omega_1^2 \cdots \omega_k^2 \ d\omega.
	\end{equation}
	Since $k$ is fixed, by an elementary argument we get $|I|=d!/(d-k)!\sim d^k$. Thus it remains to show that
	\begin{equation}
		\label{eq:lemA1'}
		\int_{S^{d-1}} \omega_1^2 \cdots \omega_k^2 \ d\omega \sim d^{-k}
	\end{equation}
	
	Formula \eqref{eq:lemA1'} is given in \cite[(10)]{sykora}. It can be also easily
	computed by the method from \cite[Chapter 3.4]{Ho}; for the sake of completeness we provide a brief argument.
	Consider the integral $J=\int_{\R^d}x_1^2...x_k^2e^{-|x|^2}dx$. Since $J$ is a product of the one-dimensional integrals we calculate $J=\Gamma \left(\frac{3}{2} \right)^k \Gamma \left(\frac{1}{2}\right)^{d-k},$ while using polar coordinates
	gives $J=S_{d-1}\int_{S^{d-1}} \omega_1^2 \cdots \omega_k^2 \ d\omega\int_0^\infty r^{2k+d-1}e^{-r^2}dr$, where $S_{d-1}$ is defined by \eqref{eq:Sd-1}. 
	Altogether we have justified that 
	\[
	\int_{S^{d-1}} \omega_1^2 \cdots \omega_k^2 \ d\omega\sim  \frac{ \Gamma \left(\frac{1}{2} \right)^{d-k}}{S_{d-1}\Gamma\left( k+\frac{d}{2} \right)}.
	\]
	Since $k$ is fixed and $d$ is arbitrarily large, using \eqref{eq:Sd-1}, Stirling's formula for the $\Gamma$ function \eqref{StirF}
	and the known identity  $\Gamma(1/2)=\sqrt{\pi}$  we obtain 
	\begin{align*}
		\int_{S^{d-1}} \omega_1^2 \cdots \omega_k^2 \ d\omega &\sim \frac{ \sqrt{k+\frac{d}{2}} \left( \frac{d}{2e} \right)^{d/2} }{\sqrt{\frac{d}{2}} \left( \frac{k+\frac{d}{2}}{e} \right)^{k+d/2}} \\
		&\sim \frac{ e^{-d/2}}{e^{-k-d/2}} \left( \frac{k+\frac{d}{2}}{d/2} \right)^{-d/2} \left( k+\frac{d}{2} \right)^{-k} \\
		&\sim   d^{-k}
	\end{align*}
	This gives \eqref{eq:lemA1'} and concludes the proof of \eqref{eq:lemA1}.

	Let now $m^t$ be the multiplier symbol of $M^t_k.$ Then, from \thref{pro:fact}  we see that $m^t=\hat{b^t}$ is radial, so that
	\begin{align*}
		m^t(\xi)&=\tilde{a}^{-1} \tilde{a}\, m^t(\xi)=\tilde{a}^{-1} \int_{SO(d)} m^t(\xi) \, a(U\xi) \, d\mu(U)\\
		&=\tilde{a}^{-1} \int_{SO(d)} m^t(U\xi) \, a(U\xi) \, d\mu(U).
	\end{align*} 
	Using properties of the Fourier transform the above equality implies that
	\begin{align*}
		M^t_k f(x)=\tilde{a}^{-1}\int_{SO(d)}\, [(M^t_k A)_U](f)(x)\,d\mu(U).
	\end{align*}
	Recalling \eqref{eq:Adef} we apply \eqref{eq:fact} from \thref{pro:fact}   and obtain
	\[
	M^t_k A=\sum_{j\in I} M^t_k R_j R_j=\sum_{j\in I}R_j^t R_j=R^t;
	\]
	here an application of \eqref{eq:fact}  is allowed since each $R_j$ corresponds to the monomial $x_j$  which is in $\mH_k$ when $j\in I.$   In summary, we justified that
	\begin{equation}
		\label{eq:Mtexpp}
		M^t_k f(x)=\tilde{a}^{-1}\int_{SO(d)}\, [(R^t)_U](f)(x)\,d\mu(U),\qquad f\in\mD(k),
	\end{equation}
	which is \eqref{eq:lemAplem} with $C(d,k)=\tilde{a}^{-1}.$ 
	
	It remains to justify \eqref{eq:lemA}. This follows from \eqref{eq:MtmaxQ}, \eqref{eq:Mtexpp}, and \eqref{eq:lemA1}, together with the norm inequality
	\[
	\norm{\int_{SO(d)}\, F_{s,t}(U)\,d\mu(U)}_X\leqslant \int_{SO(d)}\,\norm{ F_{s,t}(U)}_X\,d\mu(U);
	\]
	on the Banach space $X=\ell^2(\{1,\ldots,S\};\ell^{\infty}(\Q_+)),$ with $F_{s,t}(U)=(R^t)_U(f_s)(x)$ and $x$ being fixed. 
	
	The proof of \thref{pro:av} is thus completed.
\end{proof}

Since conjugation by $U\in SO(d)$  is an isometry on all $L^p$ spaces, in view of $\mu(SO(d))=1$ and Minkowski's integral inequality  \thref{pro:av} eq.\ \eqref{eq:lemA} allows us to deduce \thref{thm1',thm2'} from the two theorems below.

\begin{theorem}
	\thlabel{thm1''}
	Fix $k\in \N.$ For each $p \in (1, \infty)$ there is a constant $A(p,k)$ independent of the dimension $d$ and such that for any $S\in\N$  we have
	\begin{equation*} 
		\norm{\left(\sum_{s=1}^S |R^* f_s|^2\right)^{1/2}}_{p} \lesssim A(p,k) \norm{ \left(\sum_{s=1}^S | f_s|^2\right)^{1/2}}_{p},
	\end{equation*}
where $f_1,\ldots,f_S \in L^p.$ Moreover, $A(p,k)$ satisfies $A(p,k)\lesssim_k (p^*)^{5/2+k/2}.$  
\end{theorem}
\begin{theorem} \thlabel{thm2''}
	Fix $k\in \N.$ For each $p \in (1, \infty)$ there is a constant $B(p,k)$ independent of the dimension $d$ and such that
	\begin{equation*} 
		\norm{ R^* f}_{p} \lesssim B(p,k)  \norm{  f}_{p}.
	\end{equation*}
whenever $f\in L^p.$ Moreover, $B(p,k)$ satisfies $B(p,k)\lesssim_k (p^*)^{2+k/2}.$   
\end{theorem}

\section{Extension to $\CC^d$ and the complex method of rotations}

\label{sec:cmr}
Here we extend the operators $R^t$ acting on $L^p(\R^d)$ to the operators $\tR^t$ acting on $L^p(\CC^d).$ Then we apply the complex method of rotations of Iwaniec and Martin \cite{iwaniec_martin} to $\tR^t$. 

Let $P\in \mH_k.$ For $z=(x_1+iy_1,\ldots,x_d+iy_d),$ $x\in \R^d,$ $y\in \R^d$ we denote
\begin{equation}
	\label{eq:KPt}
	\tK_P(z)= \tg_k \frac{P(z)}{\abs{z}^{2d+k}} \qquad\textrm{ with } \qquad  \tg_k = \frac{\Gamma\left(d+ \frac{k}{2}\right)}{\pi^{d}\Gamma\left( \frac{k}{2}\right)},
\end{equation}
and define, for $f\in \mS(\CC^d),$ 
\begin{equation} \label{eq:tR}
	\tR_P f(z) = \lim_{t \to 0} \tR_P^t f(z), \qquad\textrm{ where } \qquad \tR_P^t f(z) = \tg_k \int_{w\in \CC^d\colon\abs{w}>t} \frac{P(w)}{ \abs{w}^{2d+k}} f(z-w) dw.
\end{equation}

In \cite{iwaniec_martin} the authors considered the extension on the multiplier level whereas we need to write it on the kernel level. This makes no difference for the operator $\tR_P.$ However, the multiplier symbol corresponding to $\tR_P^t$ does not have a simple formula, thus writing the extension on a kernel level seems the only reasonable option here. 

Formulas \eqref{eq:KPt} and \eqref{eq:tR} lead us to define the extension of $R^t$ by
\begin{equation}
	\label{eq:tRt}
	\tR^t=\tR_k^t:= \sum_{j\in I} \tR_{j}^t \tR_{j}.
\end{equation}
Using the complex method of rotations  \cite[Section 6]{iwaniec_martin} we will prove $L^p(\CC^d)$ estimates for 
\begin{equation*} 
	\tR^*f(z)=\sup_{t\in\mathbb{Q}_+}|\tR^tf(z)|.
\end{equation*}

\begin{theorem} \thlabel{thm1c''}
	Fix $k\in \N.$ For each $p\in (1,\infty)$  there is a constant $A(p,k)$ independent of the dimension $d$ and such that for any $S\in\N$  we have
	\begin{equation*} 
		\norm{\left(\sum_{s=1}^S |\tR^* f_s|^2\right)^{1/2}}_{L^p(\CC^d)} \leqslant A(p,k) \norm{ \left(\sum_{s=1}^S | f_s|^2\right)^{1/2}}_{L^p(\CC^d)},
	\end{equation*}
	whenever $f_1,\ldots,f_S \in L^p(\CC^d).$ Moreover,  $A(p,k)$ satisfies $A(p,k) \lesssim_k (p^*)^{5/2+k/2}$.
\end{theorem}
\begin{theorem} \thlabel{thm2c''}
	Fix $k\in \N.$ For each $p\in (1,\infty)$  there is a constant $ B(p,k)$ independent of the dimension $d$ and such that
	\begin{equation*} 
		\norm{ \tR^* f}_{L^p(\CC^d)} \leqslant B(p,k) \norm{  f}_{L^p(\CC^d)},
	\end{equation*}
	whenever $f\in L^p(\CC^d).$ Moreover,  $B(p,k)$ satisfies $B(p,k) \lesssim_k (p^*)^{2+k/2}$.
\end{theorem}
The reminder of this section will be devoted to the proofs of \thref{thm1c''} and \thref{thm2c''}. From these results we shall obtain \thref{thm1'} and \thref{thm2'} provided we develop a restriction procedure from $\CC^d$ to $\R^d.$ As we already remarked this is not straightforward, since the restriction of the complex truncated Riesz transform is not the real truncated Riesz transform. Details of the restriction and estimates for the resulting operators are given in Section \ref{sec:re}.

We now focus on the proofs of \thref{thm1c''} and \thref{thm2c''}. Let $P\in \mH_k$. Note that
\[
2\pi \int_{\CC^{d}} F(w) \, dw = \int_{S^{2d-1}} \int_\CC F(\lambda \theta) \abs{\lambda}^{2d-2} \, d\lambda \, d\theta,
\]
where $F\in \mS(\CC^d)$ and $d\theta$ stands for the spherical measure on $S^{2d-1}$ normalized by the condition $\theta(S^{2d-1})=S_{2d-1}.$
Take $f\in \mS(\CC^d).$ Applying the above identity with $F(w) = \widetilde{\gamma}_k \frac{P(w)}{\abs{w}^{2d+k}} \ind{|w|\ge t} f(z-w)$ gives
\begin{align*}
	\widetilde{R}_P^t f(z) &= \widetilde{\gamma}_k \int_{\CC^d} \frac{P(w)}{\abs{w}^{2d+k}} \ind{|w|\ge t} f(z-w) \, dw \\
	&=\frac{ \widetilde{\gamma}_k}{2\pi} \int_{S^{2d-1}} \int_\CC \frac{P(\lambda \theta)}{\abs{\lambda}^{2d+k}} \ind{|\la|\ge t} f(z-\lambda \theta) \abs{\lambda}^{2d-2} \, d\lambda \, d\theta \\
	&= \frac{ \widetilde{\gamma}_k}{2\pi} \int_{S^{2d-1}} P(\theta) \int_\CC \left( \frac{\lambda}{\abs{\lambda}} \right)^k \frac{f(z-\lambda \theta)}{\abs{\lambda}^2} \ind{|\la|\ge t} \, d\lambda \, d\theta,
\end{align*}
where in the last equality above we used the $k$-homogeneity of $P$. 
This means that we got
\begin{equation} \label{eq:rot}
	\widetilde{R}_P^t f(z) = \frac{\widetilde{\gamma}_k}{2\pi} \int_{S^{2d-1}} P(\theta) H_{\theta,k}^t f(z) \, d\theta,
\end{equation}
where
\[
H_{\theta,k}^t f(z)=H_\theta^t f(z) := \int_\CC \left( \frac{\lambda}{\abs{\lambda}} \right)^k \frac{f(z-\lambda \theta)}{\abs{\lambda}^2} \ind{|\la|\ge t}(\lambda) \, d\lambda
\]
is the truncated directional $k$-th power of the complex Hilbert transform. Identity \eqref{eq:rot} can be written in terms of the probabilistic spherical measure $d\zeta$ on $S^{2d-1}$  in the following way
\begin{equation} \label{eq:rot2}
	\widetilde{R}_P^t f(z)= \frac{\Gamma\left( d+\frac{k}{2} \right)}{ \pi \Gamma\left( d \right) \Gamma\left( \frac{k}{2} \right)} \int_{S^{2d-1}} P(\zeta) H_\zeta^t f(z) \, d\zeta.
\end{equation}
The limiting case of \eqref{eq:rot2} is then
\begin{equation} \label{eq:rot3}
	\widetilde{R}_P f(z)= \frac{\Gamma\left( d+\frac{k}{2} \right)}{ \pi \Gamma\left( d \right) \Gamma\left( \frac{k}{2} \right)} \int_{S^{2d-1}} P(\zeta) H_\zeta f(z) \, d\zeta,
\end{equation}
where $$H_\zeta f=H_{\zeta,k} f={\rm p.v.\,}\int_\CC \left( \frac{\lambda}{\abs{\lambda}} \right)^k \frac{f(z-\lambda \zeta)}{\abs{\lambda}^2} \, d\lambda$$ is the directional $k$-th power of the complex Hilbert transform. Identities \eqref{eq:rot2} and \eqref{eq:rot3} were initially established for $f\in \mS(\CC^d).$ However, a density argument based on the $L^p(\CC^d)$ boundedness of $ H_\zeta^t$ and $ H_\zeta$ allows us to write these identities for all $f\in L^p(\CC^d).$ For further reference we note that when $k$ is fixed then
\begin{equation}
	\label{eq:agkd}
	 \frac{\Gamma\left( d+\frac{k}{2} \right)}{ \pi \Gamma\left( d \right) \Gamma\left( \frac{k}{2} \right)}\sim d^{k/2}.
\end{equation} 

In the proofs of \thref{thm1''} and \thref{thm2''} we shall need boundedness properties of the maximal operator
\[
	H_{\zeta}^*f(z)=H_{\zeta,k}^*f(z):=\sup_{t\in \mathbb{Q}_+} |H_{\zeta}^tf(z)|
\]
 associated to $H_{\zeta}^t.$ 
\begin{pro}
	\thlabel{pro:hzt}
	For each $1<p<\infty$  we have
\begin{equation*}
	\norm{\left(\sum_{s=1}^S |H_{\zeta}^* f_s|^2\right)^{1/2}}_{L^p(\CC^d)} \lesssim p^* \norm{ \left(\sum_{s=1}^S | f_s|^2\right)^{1/2}}_{L^p(\CC^d)}
\end{equation*}	
uniformly in $\zeta\in S^{2d-1}$ and the dimension $d.$  
\end{pro} 
\noindent The proof of \thref{pro:hzt} is standard therefore we omit it here. For the convenience of the reader we include the proof in the Appendix  \ref{sec:app}.

We will also need vector-valued estimates for $\{\tR_j(f_s)\},$ $j\in I,$ $s=1,\ldots,d.$
\begin{pro}
	\thlabel{pro:tRve}
	Fix $k\in \N.$ Then for each $1<p<\infty$ we have 
	\begin{equation}
		\label{eq:tRve}
		\norm{\left(\sum_{s=1}^S \sum_{j\in I} |\tR_{j}f_s|^2\right)^{1/2}}_{L^p(\CC^d)} \lesssim_k p^* p^{1/2} q^{\frac{k+1}{2}} \norm{ \left(\sum_{s=1}^S | f_s|^2\right)^{1/2}}_{L^p(\CC^d)},
	\end{equation}	
	\begin{equation}
		\label{eq:tRvs}
		\norm{\left( \sum_{j\in I} |\tR_{j}f|^2\right)^{1/2}}_{L^p(\CC^d)} \lesssim_k p^* q^{k/2} \norm{f}_{L^p(\CC^d)},
	\end{equation}	
	uniformly in the dimension $d$.
\end{pro} 
\thref{pro:tRve} can be proved by an iterative application of its $k=1$ case together with Khintchine's inequalities. However, such an approach produces worse constants than those in  \eqref{eq:tRve}, \eqref{eq:tRvs}. An important ingredient in the proof are properties of the functions $$\zeta_j=(x_{j_1}+iy_{j_1}) \cdots (x_{j_k}+iy_{j_k}).$$ Note that $\zeta_j,$ $j\in I,$ are orthogonal with respect to the inner product on $S^{2d-1}.$ Moreover,
	\begin{equation}
		\label{eq:zjnorm}
		\int_{S^{2d-1}} |\zeta_j|^2\,d\zeta\lesssim d^{-k}.
	\end{equation}
	Indeed, all the integrals on the left hand side of \eqref{eq:zjnorm} are equal for $j\in I$ and thus
	\begin{align*}
		\int_{S^{2d-1}} |\zeta_j|^2\,d\zeta&=\frac{1}{|I|}\int_{S^{2d-1}}\sum_{j\in I} |\zeta_j|^2\,d\zeta \leq \frac{1}{|I|}\int_{S^{2d-1}}\sum_{j\in [d]^k} |\zeta_j|^2\,d\zeta\\
		&= \frac{1}{|I|} \int_{S^{2d-1}}|\zeta|^{2k}\,d\zeta \lesssim d^{-k},
	\end{align*}
	since $|I|\approx d^{k}.$
	
	We justify \eqref{eq:tRve} and \eqref{eq:tRvs} separately, starting with the latter.
	
\begin{proof}[Proof of \eqref{eq:tRvs}] Take numbers $\lambda_j(f,z)=\lambda_j(z),$ $j\in I$, such that
	\begin{equation*} 
		\left( \sum_{\vj \in I} \abs{ \tR_j f(z) }^2 \right)^{1/2} = \sum_{\vj \in I} \lambda_j(z)  \tR_jf(z), \qquad \sum_{\vj \in I} \lambda_j^2(z) = 1.
	\end{equation*}
	Using \eqref{eq:rot3} and \eqref{eq:agkd} followed by H\"older's inequality we obtain
	\begin{align} \label{eq8}
		&\norm{	\left( \sum_{\vj \in I} \abs{\tR_j f}^2 \right)^{1/2} }_p^p = \int_{\CC^d} \abs{\sum_{\vj \in I} \lambda_j(z)  \tR_jf(z)}^p dz \nonumber \\
		&\lesssim^p  d^{kp/2}\int_{\CC^d} \abs{\int_{S^{2d-1}} \sum_{\vj \in I} \lambda_{j}(z) \zeta_j H_\zeta f(z) d\zeta }^p dz \nonumber \\
		&\leqslant d^{kp/2} \int_{\CC^d} \left(\int_{S^{2d-1}} \abs{\sum_{\vj \in I} \lambda_j(z) \zeta_j}^q d\zeta \right)^{p/q} \int_{S^{2d-1}} \abs{ H_\zeta f(z)}^p d\zeta dz;
	\end{align}

Now we deal with the first inner integral in \eqref{eq8}.  Since $\zeta_j\in \mH_k^{2d} $ for $j\in I,$  for fixed $z$ the function  $\zeta\mapsto \sum_{j\in I} \zeta_j \lambda_j(z)$ also belongs to $\mH_k^{2d}.$  Using \cite[Lemme, p. 195]{duo_rubio}, orthogonality of the functions $\zeta_j,$ $j\in I,$ inequality \eqref{eq:zjnorm}, and the formula ${\sum_{j\in I}\lambda_j(z)^2=1}$ we get
\begin{equation}
	\label{eq:calcul}
	\begin{split}
		&\left(\int_{S^{2d-1}} \abs{\sum_{\vj \in I} \lambda_j(z) \zeta_j}^q d\zeta \right)^{1/q} \lesssim q^{k/2} \left(\int_{S^{2d-1}} \abs{\sum_{\vj \in I} \lambda_j(z) \zeta_j}^2 d\zeta \right)^{1/2} \\
		&= q^{k/2}  \left(\int_{S^{2d-1}} \sum_{\vj \in I} \lambda_j(z)^2 \abs{\zeta_j}^2 d\zeta \right)^{1/2} \lesssim q^{k/2} \left(d^{-k}\sum_{\vj \in I} \lambda_j(z)^2 \right)^{1/2}\leq q^{k/2} d^{-k/2}.
	\end{split}
\end{equation}

Applying \eqref{eq:calcul} and coming back to \eqref{eq8} we obtain
\begin{align*}
	\norm{	\left( \sum_{\vj \in I} \abs{\tR_j f}^2 \right)^{1/2} }_p\lesssim q^{k/2}\left( \int_{S^{2d-1}} \norm{ H_\zeta f}_{L^p(\CC^d)}^p\,d\zeta\right)^{1/p}.
\end{align*}
Now \thref{pro:hzt} completes the proof of \eqref{eq:tRvs}.
\end{proof}

We are now ready to prove \eqref{eq:tRve}. This is similar to the proof of \eqref{eq:tRvs} with an addition of Khintchine's inequalities. For $s=1,2,\ldots$  we let $\{r_s\}$ be the Rademacher functions, see \cite[Appendix C]{grafakos}. These form an orthonormal set on $L^2([0,1])$.
Moreover we have Khintchine's inequalities (\cite[Appendix C.2]{grafakos})
\begin{equation} 
	\label{eq:chinczyn}
	\norm{\sum_{j=1}^\infty a_j r_j}_{L^p([0,1])} \lesssim p^{\frac12} \left( \sum_{j=1}^\infty \abs{a_j}^2 \right)^{1/2}
\end{equation}
and
\begin{equation}
	\label{eq:chinczyn'} 
	\left( \sum_{j=1}^\infty \abs{a_j}^2 \right)^{1/2}\lesssim \norm{\sum_{j=1}^\infty a_j r_j}_{L^p([0,1])} 
\end{equation}
for any complex sequence $(a_s)_{s=1}^\infty$ and $1 \leqslant  p < \infty.$ The explicit bounds on constants in  \eqref{eq:chinczyn} and \eqref{eq:chinczyn'} follow from explicit values of the optimal constants established by Haagerup \cite{Ha} together with Stirling's formula \eqref{StirF}.  
	\begin{proof}[Proof of \eqref{eq:tRve}]
		Take numbers $\lambda_{j,s}(z,\{f_s\})=\lambda_{j,s}(z),$ $j\in I$, $s=1,\ldots,S,$ such that
		\begin{equation} \label{eq7'}
			\left( \sum_{\vj \in I}\sum_{s=1}^S \abs{ \tR_j f_s(z) }^2 \right)^{1/2} = \sum_{s=1}^S\sum_{\vj \in I} \lambda_{j,s}(z)  \tR_jf_s(z), \qquad \sum_{s=1}^S\sum_{\vj \in I} \lambda_{j,s}^2(z)= 1.
		\end{equation}
	Using \eqref{eq7'}, \eqref{eq:rot3}, and \eqref{eq:agkd} we obtain
\begin{align} \label{eq8'}
		&\norm{\left( \sum_{\vj \in I}\sum_{s=1}^S \abs{\tR_j f_s}^2 \right)^{1/2}}_p^p = \int_{\CC^d} \abs{\sum_{s=1}^S \sum_{\vj \in I} \lambda_{j,s}(z)  \tR_jf_s(z)}^p dz \nonumber \\
		&\lesssim^p  d^{kp/2}\int_{\CC^d} \abs{\int_{S^{2d-1}} \sum_{s=1}^S\sum_{\vj \in I} \lambda_{j,s}(z) \zeta_j H_\zeta f_s(z) d\zeta }^p dz.
	\end{align}
Orthogonality of the Rademacher functions $\{r_s\}$ and H\"{o}lder's inequality imply
	\begin{equation}
		\label{eq:rad1}
		\begin{split}
	&d^{kp/2}\int_{\CC^d} \abs{\int_{S^{2d-1}} \sum_{s=1}^S\sum_{\vj \in I} \lambda_{j,s}(z) \zeta_j H_\zeta f_s(z) d\zeta }^p dz\\
	&=d^{kp/2}	\int_{\CC^d} \abs{\int_{S^{2d-1}}\int_0^1 \left(\sum_{s=1}^S\sum_{\vj \in I} r_s(\xi)\lambda_{j,s}(z) \zeta_j \right)\left(\sum_{s=1}^S r_s(\xi) H_\zeta f_s(z)\right) d\xi\,d\zeta }^p dz\\
	&\leq d^{kp/2}\int_{\CC^d} \left( \int_{S^{2d-1}} \int_0^1 \abs{\sum_{s=1}^S\sum_{\vj \in I} r_s(\xi)\lambda_{j,s}(z) \zeta_j }^q d\xi \, d\zeta \right)^{p/q}\\
	&\hspace{2.5cm} \times \int_{S^{2d-1}} \int_0^1\abs{\sum_{s=1}^S r_s(\xi) H_\zeta f_s(z)}^p\,d\xi\,d\zeta\,dz.
	\end{split}
	\end{equation}
Denote
\begin{equation*}
	Q_{S,q}(z):=\left(\int_{S^{2d-1}} \int_0^1 \abs{\sum_{s=1}^S\sum_{\vj \in I} r_s(\xi)\lambda_{j,s}(z) \zeta_j }^q  d\xi d\zeta\right)^{1/q}
\end{equation*}
Then, coming back to \eqref{eq8'} and using Khintchine's inequality \eqref{eq:chinczyn} to the second factor in the last inequality in \eqref{eq:rad1} we reach
\begin{equation*}
	\begin{split}
		&\norm{\left( \sum_{\vj \in I}\sum_{s=1}^S \abs{\tR_j f_s}^2 \right)^{1/2}}_p^p \lesssim^p p^{p/2} d^{kp/2}\norm{Q_{S,q}}_{L^{\infty}(\CC^d)}^p\int_{S^{2d-1}}\int_{\CC^d} \left(\sum_{s=1}^S |H_\zeta f_s(z)|^2\right)^{p/2}  \,dz \, d\zeta.
	\end{split}
\end{equation*}
Thus, \thref{pro:hzt} implies
\begin{equation*}
		\norm{\left( \sum_{\vj \in I}\sum_{s=1}^S \abs{\tR_j f_s}^2 \right)^{1/2}}_p \lesssim p^* p^{1/2} d^{k/2}\norm{Q_{S,q}}_{L^{\infty}(\CC^d)}\norm{\left(\sum_{s=1}^S \abs{f_s}^2 \right)^{1/2}}_{L^p(\CC^d)}.
\end{equation*} 
Therefore, the proof of \eqref{eq:tRve} will be completed if we justify that
\begin{equation}
	\label{eq:Qinf}
	\norm{Q_{S,q}}_{L^{\infty}(\CC^d)}\lesssim q^{\frac{k+1}{2}} d^{-k/2}.
\end{equation}
	
	The proof of \eqref{eq:Qinf} splits into two cases. 
	
	If $q \geqslant 2$, we apply Khintchine's inequality \eqref{eq:chinczyn}, Minkowski's inequality and \cite[Lemme, p.\ 195]{duo_rubio}, 
	obtaining
	\begin{equation*} 
		\begin{aligned}
		(Q_{S,q}(z))^q&\lesssim^q q^{q/2} \int_{S^{2d-1}} \left( \sum_{s=1}^S \abs{ \sum_{\vj \in I} \lambda_{j,s}(z) \zeta_j }^2 \right)^{q/2} d\zeta \\
			&\leqslant q^{q/2} \left( \sum_{s=1}^S \left( \int_{S^{2d-1}} \left| \sum_{\vj \in I} \lambda_{j,s}(z) \zeta_j \right|^q d\zeta \right)^{2/q} \right)^{q/2} \\
			&\lesssim^q q^{q/2} q^{kq/2} \left(  \sum_{s=1}^S \int_{S^{2d-1}} \abs{ \sum_{\vj \in I} \lambda_{j,s}(z) \zeta_j }^2 d\zeta \right)^{q/2},
		\end{aligned}
	\end{equation*}
uniformly in $z\in \CC^d.$
	Here an application of \cite[Lemme, p.\ 195]{duo_rubio} is justified since $\zeta_j\in \mH_k^{2d}$ for  $j\in I$ and thus also the sum   $\sum_{\vj \in I} \lambda_{j,s}(z) \zeta_j$ belongs to $\mH_k^{2d}$ for each fixed $z\in \CC^d$. Now, using the orthogonality of $\zeta_j,$ $j\in I,$ inequality \eqref{eq:zjnorm} and the formula $\sum_{s=1}^S\sum_{\vj \in I} \lambda_{j,s}^2(z)= 1$ we see that
	\begin{equation*}
		\begin{split}
		(Q_{S,q}(z))^q
			&\lesssim^q q^{q/2} q^{kq/2}\left( \sum_{s=1}^S \int_{S^{2d-1}} \sum_{\vj \in I} \lambda_{j,s}(z)^2 \abs{\zeta_j}^2 \, d\zeta \right)^{q/2}\\
			&=  q^{q/2} q^{kq/2} \left(d^{-k} \sum_{s=1}^S \sum_{\vj \in I} \lambda_{j,s}(z)^2 \right)^{q/2}\lesssim q^{q/2} q^{kq/2} d^{-kq/2}.
		\end{split}
	\end{equation*}
Therefore, \eqref{eq:Qinf} is justified in the case $q \geqslant 2.$   
	
	If on the other hand $1 < q < 2$, an application of H\"older's inequality together with \eqref{eq:Qinf} in the case $q=2$ shows that
	\begin{align*}
		Q_{S,q}(z)\leq Q_{S,2}(z)\lesssim d^{-k/2}.
	\end{align*}
This completes the proof of \eqref{eq:Qinf} and thus also the proof of \eqref{eq:tRve} from \thref{pro:tRve}.
		
	\end{proof}

We are now ready to prove \thref{thm1c''} and \thref{thm2c''}. In both the proofs we shall need the formula
\begin{equation}
	\label{eq:tRtfHz}
\tR^t f(z)= \frac{\Gamma\left( d+\frac{k}{2} \right)}{ \pi \Gamma\left( d \right) \Gamma\left( \frac{k}{2} \right)}  \int_{S^{2d-1}} H_\zeta^t\left[\sum_{j\in I} \zeta_j \tR_j f \right](z)\, d\zeta,
\end{equation}
which follows from  \eqref{eq:tRt} and \eqref{eq:rot2}. We start with the proof of \thref{thm2c''}.
\begin{proof}[Proof of \thref{thm2c''}]
\label{p:pfthm2c''}
Using \eqref{eq:tRtfHz} and \eqref{eq:agkd} we see that
\[
|\tR^* f(z)|\lesssim d^{k/2} \int_{S^{2d-1}} H_\zeta^*\bigg[\sum_{j\in I} \zeta_j \tR_j f \bigg](z)\, d\zeta,\qquad z\in \CC^d.\]
Hence, Minkowski's integral inequality followed by \thref{pro:hzt} show that
\begin{align*}
\norm{\tR^* f}_{L^p(\CC^d)}\lesssim p^* d^{k/2} \int_{S^{2d-1}}\norm{ \sum_{j\in I} \zeta_j \tR_j f }_{L^p(\CC^d)}\,d\zeta.
\end{align*}
Using H\"older's inequality and Fubini's theorem we obtain
\begin{equation}
	\label{eq:Rt1}
\norm{\tR^* f}_{L^p(\CC^d)}\lesssim p^* d^{k/2}\left(\int_{\CC^d}\int_{S^{2d-1}}\left|\sum_{j\in I} \zeta_j \tR_j f(z) \right|^p\,d\zeta\,dz\right)^{1/p}.
\end{equation}
Since for fixed $z$ the function  $\zeta\mapsto \sum_{j\in I} \zeta_j \tR_j f(z)$ belongs to $\mH_k^{2d}$, applying \cite[Lemme, p. 195]{duo_rubio} we obtain
	\begin{equation*}
	\begin{split}
		\left(\int_{S^{2d-1}} \abs{\sum_{j \in I} \zeta_j \tR_j f(z)}^p d\zeta\right)^{1/p} \lesssim p^{k/2} \left(\int_{S^{2d-1}} \abs{\sum_{j \in I} \zeta_j \tR_j f(z)}^2 \,d\zeta\right)^{1/2}.
	\end{split}
\end{equation*}
Using orthogonality and \eqref{eq:zjnorm} we thus see that
	\begin{equation}
	\label{eq:calcu'}
		\left(\int_{S^{2d-1}} \abs{\sum_{j \in I} \zeta_j \tR_j f(z)}^p d\zeta\right)^{1/p} \lesssim d^{-k/2}\,p^{k/2}  \left(\sum_{j \in I} | \tR_j f(z)|^2 \right)^{1/2},
\end{equation}
which, together with \eqref{eq:Rt1} leads to
\[
	\norm{ \tR^* f}_{L^p(\CC^d)} \lesssim p^* p^{k/2} \norm{\left(\sum_{j \in I} | \tR_j f|^2 \right)^{1/2}}_{L^p(\CC^d)}.
\]
Thus, \eqref{eq:tRvs} from \thref{pro:tRve} completes the proof of \thref{thm2c''}.

\end{proof}

We finish this section with the proof of \thref{thm1c''}. 
\begin{proof}[Proof of \thref{thm1c''}]
	\label{p:pfthm1c''}
	Using \eqref{eq:tRtfHz}, \eqref{eq:agkd}, and Minkowski's integral inequality on the space $\ell^2(\{1,...,S\};L^{\infty}(\Q_+))$  we see that
	\[
		\left(\sum_{s=1}^S |\tR^* f_s(z)|^2\right)^{1/2}\lesssim d^{k/2} \int_{S^{2d-1}} \left(\sum_{s=1}^S\bigg(H_\zeta^*\bigg[\sum_{j\in I} \zeta_j \tR_j f_s \bigg](z)\bigg)^2\right)^{1/2}\, d\zeta,\qquad z\in \CC^d.
	\]
	Thus, another application of Minkowski's integral inequality followed by \thref{pro:hzt} gives 
	\[
		\norm{\left(\sum_{s=1}^S |\tR^* f_s|^2\right)^{1/2}}_{L^p(\CC^d)}\lesssim p^* d^{k/2} \int_{S^{2d-1}} \norm{\left(\sum_{s=1}^S\bigg|\sum_{j\in I} \zeta_j \tR_j f_s \bigg|^2\right)^{1/2}}_{L^p(\CC^d)}\, d\zeta.
	\]
	Using Khintchine's inequality \eqref{eq:chinczyn'} followed by H\"older's inequality on $S^{2d-1}$ we see that 
	\begin{align*}
		&\norm{\left(\sum_{s=1}^S |\tR^* f_s|^2\right)^{1/2}}_{L^p(\CC^d)}\\
		&\lesssim p^* d^{k/2} \int_{S^{2d-1}} 	\left(\int_{\CC^d}\int_0^1 \bigg|\sum_{s=1}^S 	r_s(\xi)\sum_{j\in I}  \zeta_j \tR_j f_s(z)\bigg|^p\,d\xi\,dz\right)^{1/p}\,d\zeta\\
		& \lesssim p^* d^{k/2} \left(\int_{\CC^d}\int_0^1 \int_{S^{2d-1}}\left|\sum_{j\in I} \zeta_j \tR_j\bigg[\sum_{s=1}^S r_s(\xi)  f_s(z)\bigg]\right|^p\,d\zeta\,d\xi\,dz\right)^{1/p}.
	\end{align*}
	Finally, \eqref{eq:calcu'} followed by \eqref{eq:tRvs} from \thref{pro:tRve} and Khintchine's inequality \eqref{eq:chinczyn} give
	\begin{align*}
		&\norm{\left(\sum_{s=1}^S |\tR^* f_s|^2\right)^{1/2}}_{L^p(\CC^d)}\lesssim p^*p^{k/2} 	\left(\int_{\CC^d} \int_0^1 \left(\sum_{j\in I}\bigg|\tR_j\bigg[\sum_{s=1}^S r_s(\xi)  f_s(z)\bigg]\bigg|^2\right)^{p/2} \, d\xi \, dz \right)^{1/p}\\
		&\lesssim (p^*)^{2+k/2} \left( \int_{\CC^d} \int_0^1 \left|\sum_{s=1}^S r_s(\xi)  	f_s(z)\right|^p \, d\xi \, dz \right)^{1/p}  \lesssim (p^*)^{5/2+k/2} \left( \int_{\CC^d}\left(\sum_{s=1}^S | f_s|^2\right)^{p/2}\,dz\right)^{1/p}.
	\end{align*}
	The proof of \thref{thm1c''} is thus completed.
	
\end{proof}

\section{Restriction to the initial Riesz transforms}
\label{sec:re}
The purpose of this section is twofold. Firstly, we restrict the maximal operator $\tR^*$ acting on $L^p(\CC^d)$ to a maximal operator $\mR^*$ acting on $L^p(\R^d).$ This is done in a way which preserves estimates for the norms. However, the restricted maximal operator $\mR^*$ is not the same as $R^*.$ Therefore, we need to estimate their difference, which is done in the second part of Section \ref{sec:re}.

\subsection{Bounding the restriction $\mR^*$ of $\tR^*$.}
In the previous section in \thref{thm1c'',thm2c''}, we proved dimension-free estimates for the operator $\tR^*$ acting on $L^p(\CC^d)$. An approach similar to \cite[Chapter 4]{iwaniec_martin} leads to dimension-free estimates for the restriction of this operator to $L^p(\R^d)$ which we now describe.

To elaborate, for $x\in \R^d$ and $t>0$ we define the restricted kernel $\mK_j^t(x)$ by 
\begin{equation}
	\label{eq:mRjt}
	\begin{split}
	\mK_j^t(x)&=\widetilde{\gamma}_k S_{d-1} \frac{x_j}{\abs{x}^{d+k}} \int_{\sqrt{\frac{t^2}{\abs{x}^2}-1}}^\infty \frac{r^{d-1}}{\left( 1+r^2 \right)^{d+k/2}} \, dr,\qquad \textrm{for }|x|<t,\\
	\mK_j^t(x)&=K_j^t(x),\qquad \textrm{for }|x| \ge t. 
	\end{split}
\end{equation}
Recall that $K_j^t$ is the kernel given by \eqref{eq:KP} when $P_j(x)=x_{j_1}\cdots x_{j_k},$ $j\in I$. A short computation based on \eqref{eq:Sd-1}, \eqref{eq:gk}, and \eqref{eq:Bform} gives, for $x\neq 0,$ 
\begin{equation}
	\label{eq:Kjred}
	\begin{split}
	&	\lim_{t\to 0^+} \widetilde{\gamma}_k S_{d-1} \frac{x_j}{\abs{x}^{d+k}} \int_{\sqrt{\frac{t^2}{\abs{x}^2}-1}}^\infty \frac{r^{d-1}}{\left( 1+r^2 \right)^{d+k/2}} \, dr\\
	&=\frac{\Gamma(d+\frac{k}{2})}{\pi^{d/2}\Gamma(\frac{k}{2})\Gamma(\frac{d}{2})}\int_{0}^\infty \frac{2r^{d-1}}{\left( 1+r^2 \right)^{d+k/2}}\,dr\cdot \frac{x_j}{\abs{x}^{d+k}}=\gamma_k \frac{x_j}{\abs{x}^{d+k}}=K_j(x).
	\end{split}
\end{equation}
 For $f\in L^p(\R^d)$ we let $\mR_j^tf=f*\mK_j^t$ and define
\begin{equation*}
	\mR^tf=\sum_{j\in I} \mR_j^t R_j f.
\end{equation*}
and
\begin{equation*}
	\mR^*f=\sup_{t\in \mathbb{Q}_+}|\mR^t f|.
\end{equation*}

A transference argument  leads to the two results below. The proofs of \thref{thm1tres,thm2tres} are based on ideas from \cite[Section 4]{iwaniec_martin}. However, compared to \cite[Section 4]{iwaniec_martin} extra difficulties arise. These complications stem from the fact that we need to restrict compositions of singular integral operators instead of just one singular integral operator. Furthermore, useful formulas for the multiplier symbols of $\tR_j^t$ or $\mR_j^t$ are not available. 
\begin{theorem} \thlabel{thm1tres}
	Fix $k\in \N.$ For each $p\in (1,\infty)$  there is a constant $A(p,k)$ independent of the dimension $d$ and such that for any $S\in\N$  we have
	\begin{equation*} 
		\norm{\left(\sum_{s=1}^S |\mR^* f_s|^2\right)^{1/2}}_{L^p(\R^d)} \leqslant A(p,k) \norm{ \left(\sum_{s=1}^S | f_s|^2\right)^{1/2}}_{L^p(\R^d)},
	\end{equation*}
	whenever $f_1,\ldots,f_S \in L^p(\R^d).$ Moreover, $A(p,k)$ satisfies  $A(p,k) \lesssim_k (p^*)^{5/2+k/2}$.
\end{theorem}
\begin{theorem} \thlabel{thm2tres}
	Fix $k\in \N.$ For each $p\in (1,\infty)$  there is a constant $ B(p,k)$ independent of the dimension $d$ and such that
	\begin{equation*} 
		\norm{ \mR^* f}_{L^p(\R^d)} \leqslant B(p,k) \norm{  f}_{L^p(\R^d)},
	\end{equation*}
	whenever $f\in L^p(\R^d).$ Moreover,  $B(p,k)$ satisfies $B(p,k) \lesssim_k (p^*)^{2+k/2}$.
\end{theorem}
The restriction procedure from  \thref{thm1c'',thm2c''} to \thref{thm1tres,thm2tres} will result in the kernels $\tK_j$ and $\tK_j^t$ defined in \eqref{eq:KPt} being integrated over their imaginary component $iy$ in $\R^d.$ This is the origin of the kernel $\mK_j^t$ as the next lemma justifies.
\begin{lemma}
	\thlabel{lem:inaux}
For each $t>0$ and $x\in \R^d$ it holds
	\begin{equation}
		\label{eq:iat}
		\int_{\R^d} \tK_j^t(x+iy) \, dy = \mK_j^t(x).
	\end{equation}
\end{lemma}
\begin{proof}
	To justify \eqref{eq:iat} consider two cases: $|x|\ge t$ and $|x|<t.$ In the first case, integrating in polar coordinates in $\R^d$ and noting that $\int_{S^{d-1}} P_j(x+ir\omega)\,d\omega=P_j(x)$   
	\begin{align*}
		&\int_{\R^d} \tK_j^t(x+iy) \, dy = \int_{y\in \R^d\colon |x+iy|\ge t} \tg_k\frac{P_j(x+iy)}{ \abs{x+iy}^{2d+k}}  dy=\int_{\R^d}  \tg_k\frac{P_j(x+iy)}{ \abs{x+iy}^{2d+k}}\,dy\\
		&=\widetilde{\gamma}_k S_{d-1}  P_j(x) \int_{0}^\infty \frac{r^{d-1}}{\left( |x|^2+r^2 \right)^{d+k/2}} \, dr=\widetilde{\gamma}_k S_{d-1} \frac{P_j(x)}{\abs{x}^{d+k}} \int_{0}^\infty \frac{r^{d-1}}{\left( 1+r^2 \right)^{d+k/2}} \, dr \\
		& =K_j(x)=K_j^t(x) = \mK_j^t(x).
	\end{align*}
	In the fourth equality above we used change of the variables $r\to r|x|$ and then we used \eqref{eq:Kjred}.
	Similarly, in the second case $|x|<t$ we obtain  
	\begin{align*}
		&\int_{y\in \R^d\colon |x+iy|\ge t} \tg_k\frac{P_j(x+iy)}{ \abs{x+iy}^{2d+k}}  dy=\widetilde{\gamma}_k S_{d-1}  P_j(x) \int_{\sqrt{t^2-\abs{x}^2}}^\infty \frac{r^{d-1}}{\left( |x|^2+r^2 \right)^{d+k/2}} \, dr\\
		& = \mK_j^t(x),
	\end{align*}
where in the second equality we used the change of variable $r\to r|x|.$ Thus \eqref{eq:iat} is justified.
\end{proof}

We present only the proof of  \thref{thm2tres}. The proof of \thref{thm1tres} is similar. We merely need a technically more involved duality argument instead of \eqref{eq:thm2tres1'} below and an application of \thref{thm1c''} instead of \thref{thm2c''}.

\begin{proof}[Proof of \thref{thm2tres}]
	By Lebesgue's monotone convergence theorem we may restrict the supremum to a finite set of positive numbers $\{t_1,\ldots,t_N\},$ as long as our final estimate is independent of $N.$ Further, a density argument shows that it suffices to consider $f\in \mathcal S(\R^d).$   
	
	For $F\colon \CC^d\to \CC$ and $u>0$ we let $(\delta_u F)(x+iy)=F(x+iu y)$ and define
	\[
	\tR^{t,u}(F)(x+iy):=(\delta_{u^{-1}}\circ \tR^t\circ\delta_u) (F)(x+iy)=\tR^t(\delta_u F)(x+iu^{-1}y).
	\]
	Using \thref{thm2c''} it is straightforward to see that
	\begin{equation*}
			\norm{ \sup_{t_1,\ldots, t_N} |\tR^{t,u} F|}_{L^p(\CC^d)} \leqslant B(p,k) \norm{  F}_{L^p(\CC^d)}.
	\end{equation*}
 Note that by duality between the spaces $L^p(\CC^d;\ell^{\infty}(\{t_1,...,t_N\}))$ and $L^q(\CC^d;\ell^{1}(\{t_1,...,t_N\}))$ the above inequality can be rewritten in the following equivalent form
  \begin{align}
  		\label{eq:thm2tres1'}
	\left|\sum_{n=1}^N \langle \tR^{t_n,u}(F) ,\,G_n\rangle_{L^2(\CC^d)}\right|	\leq B(p,k) \|F\|_{L^p(\CC^d)} \norm{\sum_{n=1}^N \abs{G_{n}}}_{L^q(\CC^d)},
\end{align} 
where $G_n\in L^q(\CC^d)$ $n=1,\ldots,N.$ 

Let $\eta\in \mathcal S(\R^d)$ be a fixed function such that $\|\eta\|_{L^p(\R^d)}=1$ and take $f\in S(\R^d)$. 
Denoting $$F(x+iy):=(f\otimes \eta)(x,y)=f(x)\cdot \eta(y),\qquad x,y\in \R^d$$  we claim that
\begin{equation}
	\label{eq:thm2tresclaim}
	\lim_{u\to 0^+} \langle \tR^{t,u} F,\,G\rangle_{L^2(\CC^d)}=\langle \mR^t(f)\otimes \eta,\,G\rangle_{L^2(\CC^d)}
\end{equation}
for any function $G\in \mathcal S(\CC^d)$  and all $t>0.$

  Assume for a moment that the claim  holds. Fix $\varepsilon\in (0,1)$ and let $\psi\in \mathcal S(\R^d)$  be a function of $L^q(\R^d)$ norm $1$ and such that $|\langle \eta, \psi \rangle_{L^2(\R^d)}|\ge (1-\varepsilon)$. Take $g_n\in \mathcal S(\R^d),$ $n=1,\ldots,N.$ Then, substituting $F=f\otimes \eta$ and $G_n=g_n\otimes \psi$ in \eqref{eq:thm2tres1'} we have
  \begin{align*}
  	\left|\sum_{n=1}^N \langle \tR^{t_n,u}(f\otimes \eta) ,\,g_n\otimes \psi \rangle_{L^2(\CC^d)}\right|	\leq B(p,k) \|f\otimes \eta\|_{L^p(\CC^d)} \norm{\sum_{n=1}^N \abs{g_{n}\otimes \psi}}_{L^q(\CC^d)}.
  \end{align*}   
At this point the claim \eqref{eq:thm2tresclaim} implies
  \begin{align*}
	\left|\sum_{n=1}^N \langle \mR^{t_n}(f)   ,\,g_n \rangle_{L^2(\R^d)}\right||\langle\eta,\psi\rangle_{L^2(\R^d)} |	\leq B(p,k) \norm{f}_{L^p(\R^d)} \norm{\sum_{n=1}^N \abs{g_{n}}}_{L^q(\R^d)}.
\end{align*}  
Now, using duality between the spaces $L^p(\R^d;\ell^{\infty}(\{t_1,...,t_N\}))$ and $L^q(\R^d;\ell^{1}(\{t_1,...,t_N\}))$ together with the density of Schwartz function in $L^q(\R^d)$ we conclude that
	\begin{equation*}
(1-\varepsilon)	\norm{ \sup_{t_1,\ldots, t_N} |\mR^{t_n} f|}_{L^p(\R^d)} \leqslant B(p,k) \norm{  f}_{L^p(\R^d)}.
\end{equation*}
Since $\varepsilon\in (0,1)$ was arbitrary this completes the proof of \thref{thm2tres}.

It remains to verify the claim \eqref{eq:thm2tresclaim}. Since $\tR^t=\sum_{j\in I}\tR_j^t \tR_j$ it is easy to see  that 
\[
\tR^{t,u}F=\sum_{j\in I} \tR_j^{t,u}\tR_j^uF,
\] 
where, for $F=f\otimes \eta,$ we denote 
\begin{align*}
\tR_j^{t,u}(F)(x+iy)=\tR_j^t(\delta_u F)(x+iu^{-1}y),\qquad \tR_j^{u}(F)(x+iy)=\tR_j(\delta_u F)(x+iu^{-1}y).
\end{align*}
Thus, it is enough to justify that
\begin{equation}
	\label{eq:thm2tresclaim'}
	\lim_{u\to 0^+} \langle \tR_j^{t,u}\tR_j^uF,\,G\rangle_{L^2(\CC^d)}  =\langle (\mR^t_j R_j f)\otimes \eta ,\,G\rangle_{L^2(\CC^d)} 
\end{equation}
for $j\in I,$ $t>0,$ and $G\in \mathcal S(\CC^d).$

Fix $j\in I$ and $t>0$ and denote by $m^t$ and $m$ the multiplier symbols on $\CC^d$ corresponding to the operators $\tR_j^t$ and  $\tR_j,$ respectively. 
Then  $\delta_u (m^t)$ and $\delta_u (m)$  are the multiplier symbols corresponding to the operators $\tR_j^{t,u}$ and $\tR_j^u,$ respectively. Thus, identifying $\CC^d$ with $\R^{2d}$, taking the Fourier transform on $\R^{2d}$, and using Plancherel's theorem we see that
\begin{equation}
	\label{eq:thm2tresFtra1}
\langle \tR_j^{t,u}\tR_j^uF,\,G\rangle_{L^2(\CC^d)}=
\langle \delta_u (m) \delta_u (m^t) \mF[F], \mF[G]\rangle_{L^2(\CC^d)}.
\end{equation}
By formula \eqref{eq:m} (applied on $\R^{2d}$) and definitions \eqref{eq:KPt}, \eqref{eq:tR} for $P_j(z):=z_j=z_{j_1}\cdots z_{j_k}$ we have
\[
\delta_u (m)(\xi,\tau)=(-i)^k\frac{P_j(\xi+iu\tau)}{|\xi+iu\tau|^k},
\] 
for $\xi,\tau\in \R^d.$ 
Hence, for $\xi\neq0$ and $\tau\in\R^d$ it holds  $\lim_{u\to 0^+}m(\xi,u\tau)=m(\xi,0)=(-i)^k\frac{P_j(\xi)}{|\xi|^k}.$ Another application of \eqref{eq:m} (this time on $\R^d$) shows that the function $m_0(\xi):=m(\xi,0)$ is the multiplier symbol of the operator $R_j$ acting on $L^2(\R^d).$ 

Since the operators $\tR_j^t$ and  $\tR_j$ are both bounded on $L^2(\CC^d)$ the functions $\delta_u(m)$ and $\delta_u(m^t)$  are in $L^{\infty}(\CC^d),$ uniformly in $u>0.$ Thus, coming back to \eqref{eq:thm2tresFtra1} and using Lebesgue's dominated convergence theorem we see that
\begin{align*}
	\lim_{u\to 0^+}\langle \tR_j^{t,u}\tR_j^uF,\,G\rangle_{L^2(\CC^d)}=	\lim_{u\to 0^+}\langle  \delta_u (m^t)\,\mF[F],\, \overline{m_0}\, \mF[G]\rangle_{L^2(\CC^d)},
\end{align*}
provided the limit on the right hand side exists. By definition of $m_0$ applying again Plancherel's theorem we obtain
\begin{equation}
		\label{eq:thm2tres3}
	\lim_{u\to 0^+}\langle \tR_j^{t,u}\tR_j^uF,\,G\rangle_{L^2(\CC^d)}=	\lim_{u\to 0^+}\langle \tR_j^{t,u} F,\, (R_j\otimes I)^* G\rangle_{L^2(\CC^d)},
\end{equation}
provided the limit on the right hand side exists. In the above formula $R_j\otimes I$ denotes the operator $R_j$ acting only on the  $\R^d$ coordinates of a function defined on $\CC^d$ and the adjoint is taken with respect to the inner product on $L^2(\CC^d).$ 
Now, if we justify that
\begin{equation}
	\label{eq:thm2tres4}
		\lim_{u\to 0^+}\langle \tR_j^{t,u} F,\, (R_j\otimes I)^* G\rangle_{L^2(\CC^d)}=\langle \mR_j^t(f)\otimes \eta,\, (R_j\otimes I)^* G\rangle_{L^2(\CC^d)}
\end{equation}
and use the formula
\[
\langle \mR_j^t(f)\otimes \eta,\, (R_j\otimes I)^* G\rangle_{L^2(\CC^d)}=\langle (\mR^t_j R_j f)\otimes \eta ,\,G\rangle_{L^2(\CC^d)}
\]
together with
\eqref{eq:thm2tres3}, then we will complete the proof of the claim \eqref{eq:thm2tresclaim'}.

Since the operators $\tR_j^{t,u}$ are uniformly bounded on $L^2(\CC^d)$ with respect to $u>0$ to prove \eqref{eq:thm2tres4} it suffices to show that
\begin{equation}
	\label{eq:thm2tres5}
	\lim_{u\to 0^+}\langle \tR_j^{t,u} F,\tilde{G}\,\rangle_{L^2(\CC^d)}=\langle \mR_j^t(f)\otimes \eta,\, \tilde{G}\rangle_{L^2(\CC^d)},
\end{equation}
where $\tilde{G}\in \mS(\CC^d).$ For $z=x+iy,$ $x,y\in \R^d,$  we have 
	\begin{equation}
			\label{eq:tRjtfu}
		\begin{split}
	&\tR_j^{t,u}(F)(z)=\tR_j^{t,u}(f\otimes \eta)(z) = u^{-d} \delta_{u^{-1}}(\tK_j^t)*(f\otimes \eta)(z) \\
	&=  \int_{\R^d}f(x-x')\int_{y'\in \R^d\colon |x'+iu^{-1}y'|\ge t} \tg_ku^{-d} \frac{P_j(x'+iu^{-1}y')}{ \abs{x'+iu^{-1}y'}^{2d+k}} \eta(y-y') \, dy'\, dx'\\
	&=\int_{\R^d}\int_{y'\in \R^d\colon |x'+iy'|\ge t}f(x-x')\, \tg_k\frac{P_j(x'+iy')}{ \abs{x'+iy'}^{2d+k}} \eta(y-uy') \, dy'\, dx'
	\end{split}
	\end{equation}
Moreover, a computation shows that for fixed $x\in \R^d$ and $t>0$  it holds
\begin{equation}
	\label{eq:thm2tresL1}
	f(x-x')\tg_k\frac{P_j(x'+iy')}{ \abs{x'+iy'}^{2d+k}}\ind{|x'+iy'|\ge t}\in L^1(\CC^d).
\end{equation}
Hence, taking the limit as $u\to 0^+$ in \eqref{eq:tRjtfu} and using Lebesgue's dominated convergence theorem followed by  \thref{lem:inaux}
 we obtain 
\begin{equation}
	\label{eq:thm2tres6}
	\begin{split}
	&\lim_{u\to 0^+}\tR_j^{t,u}(F)(z) = \eta(y)\int_{\R^d}f(x-x')\int_{y'\in \R^d\colon |x'+iy'|\ge t}  \tg_k\frac{P_j(x'+iy')}{ \abs{x'+iy'}^{2d+k}}\,dy'\, dx'\\
	&=\eta(y)\int_{\R^d}f(x-x')\mK_j^t(x')\, dx'=\eta(y)\mR_j^tf(x)=(\mR_j^t(f)\otimes \eta)(x,y),
	\end{split}
\end{equation}
for $x,y\in \R^d.$ Moreover, another application of \eqref{eq:thm2tresL1} shows that $\tR_j^{t,u}(F) \in L^{\infty}(\CC^d),$ uniformly in $u>0.$
 Now, since $\tilde{G}\in \mS(\CC^d)$ using
again Lebesgue's dominated convergence theorem followed by \eqref{eq:thm2tres6} we reach
\begin{equation*}
	\lim_{u\to 0^+}\langle \tR_j^{t,u} F,\tilde{G}\,\rangle_{L^2(\CC^d)}=	\langle 	\lim_{u\to 0^+}\tR_j^{t,u} F,\tilde{G}\,\rangle_{L^2(\CC^d)}=
	\langle \mR_j^t(f)\otimes \eta,\, \tilde{G}\rangle_{L^2(\CC^d)},
\end{equation*}
This justifies \eqref{eq:thm2tres5}, hence, also the claim \eqref{eq:thm2tresclaim'}. The proof of \thref{thm2tres} is thus completed.

\end{proof}

\subsection{Bounding the difference between $R^t$ and $\mR^t$}

Define the difference kernels on $\R^d$ by
\begin{equation}
	\label{eq:Djt}
	E_j^t(x):=K_j^t(x)-\mK_j^t(x).
\end{equation}
Recall that by definitions \eqref{eq:KP} of $K_j^t$ and \eqref{eq:mRjt} of $\mK_j^t$ we have $E_j^t(x)=-\mK_j^t(x)$ if $|x|<t$ and  $E_j^t(x)=0$ if $|x|\ge t.$ We let $D_j$ be the operator on $L^p(\R)$ given by $D_j^tf=f*E_j^t$ and define
\begin{equation*}
	D^tf=\sum_{j\in I} D_j^t R_j f,\qquad D^*f=\sup_{t\in \mathbb{Q}_+}|D^tf|. 
\end{equation*}
Clearly,
\[R^t=\mR^t+D^t,\]
so using \thref{thm1tres,thm2tres} we reduce \thref{thm1'',thm2''} to the following two statements. 

\begin{theorem} \thlabel{thm1quat}
	Fix $k\in \N.$ For each $p\in (1,\infty)$  there is a constant $A(p,k)$ independent of the dimension $d$ and such that for any $S\in\N$  we have
	\begin{equation*} 
		\norm{\left(\sum_{s=1}^S |D^* f_s|^2\right)^{1/2}}_{L^p(\R^d)} \leqslant A(p,k) \norm{ \left(\sum_{s=1}^S | f_s|^2\right)^{1/2}}_{L^p(\R^d)},
	\end{equation*}
	whenever $f_1,\ldots,f_S \in L^p(\R^d).$ Moreover, $A(p,k)$ satisfies $A(p,k) \lesssim_k (p^*)^{5/2+k/2}$.
\end{theorem}
\begin{theorem} \thlabel{thm2quat}
	Fix $k\in \N.$ For each $p\in (1,\infty)$  there is a constant $ B(p,k)$ independent of the dimension $d$ and such that
	\begin{equation*} 
		\norm{ D^* f}_{L^p(\R^d)} \leqslant B(p,k) \norm{  f}_{L^p(\R^d)},
	\end{equation*}
	whenever $f\in L^p(\R^d).$ Moreover, $B(p,k)$ satisfies $B(p,k) \lesssim_k (p^*)^{2+k/2}$.
\end{theorem}

The proofs of the above two theorems will follow the scheme of the proofs of \thref{thm1c'',thm2c''}. The main difference lies in the application of the method of rotations. It has to be appropriate for the operator $D^t.$ For $t>0$ we let $I^t$ be the function on $(0,\infty)$ given by  
\begin{equation}
	\label{eq:Ir}
	I^t(r)=\ind{(0,t)}(r)\int_{\sqrt{\frac{t^2}{r^2}-1}}^\infty \frac{s^{d-1}}{\left( 1+s^2 \right)^{d+k/2}} \, ds,\qquad r>0.
\end{equation}

Using the definitions \eqref{eq:mRjt} and  \eqref{eq:Djt} and integrating in polar coordinates in $\R^d$ we obtain
\begin{equation}
	\label{eq:Djtrot}
	\begin{split}
	-D_j^t f(x) &=  \int_{\R^d} \widetilde{\gamma}_k S_{d-1} \frac{y_j}{\abs{y}^{d+k}} I^t(|y|) f(x-y) \, dy \\
	&= \widetilde{\gamma}_k S_{d-1}^2 \int_0^t \int_{S^{d-1}} \frac{\omega_j}{r} I^t(r) f(x-r\omega) \, d\omega \, dr \\
	&= \gamma_k S_{d-1} \int_{S^{d-1}}\omega_j \mH_{\omega}^t f(x) \, d\omega=\frac{2\Gamma\left( \frac{k+d}{2}\right)}{\Gamma\left( \frac{k}{2}\right)\Gamma\left( \frac{d}{2}\right)}\int_{S^{d-1}}\omega_j \mH_{\omega}^t f(x) \, d\omega,
	\end{split}
\end{equation}
where
\begin{equation}
	\label{eq:mHomt}
\mH_\omega^tf(x) =  \frac{\widetilde{\gamma_k}}{\gamma_k} S_{d-1}\int_0^t I^t(r) \frac{f(x-r\omega)}{r} \, dr.
\end{equation}
Let now $\mH_{\omega}^*f(x)=\sup_{t\in \mathbb{Q}_+}|\mH_\omega^tf(x)|.$
The next proposition serves as a replacement for \thref{pro:hzt}.
\begin{pro}
	\thlabel{pro:mhzt}
	For each $1<p<\infty$ we have
	\begin{equation}
		\label{eq:mmhztv}
		\norm{\left(\sum_{s=1}^S |\mH_{\omega}^* f_s|^2\right)^{1/2}}_{L^p(\R^d)} \lesssim  p^* \norm{ \left(\sum_{s=1}^S | f_s|^2\right)^{1/2}}_{L^p(\R^d)}
	\end{equation}	
	uniformly in $\omega\in S^{d-1}$ and the dimension $d.$  
\end{pro} 
\begin{proof}
For $\omega\in S^{d-1}$ and $t>0$ we let
\[
\mM_{\omega}^tf(x)=\frac{1}{t}\int_{-t}^t |f(x-r\omega)|\,dr,\qquad \mM_{\omega}^*f(x)=\sup_{t>0}|\mM_{\omega}^tf(x)|,
\] 
be the directional Hardy--Littlewood averaging operator and the directional Hardy--Littlewood maximal function. Using Fubini's theorem and one-dimensional estimates for the Hardy--Littlewood maximal function, see e.g.\ \cite[Theorem 5.6.6]{grafakos}, we obtain
	\begin{equation*}
	\norm{\left(\sum_{s=1}^S |\mM_{\omega}^* f_s|^2\right)^{1/2}}_{L^p(\R^d)} \lesssim p^*\norm{ \left(\sum_{s=1}^S | f_s|^2\right)^{1/2}}_{L^p(\R^d)},
\end{equation*}	
uniformly in $\omega\in S^{d-1}.$ Thus, to prove \eqref{eq:mmhztv} it suffices to show  the pointwise estimate 
\[
\mH_\omega^tf(x)\lesssim \mM_{\omega}^tf(x)
\]
uniformly in $x\in \R^d,$ $\omega\in S_{d-1},$ with in-explicit constants independent of the dimension.  

This bound will follow if we justify that
\begin{equation}
	\label{eq:mHmM}
	\frac{\widetilde{\gamma_k}}{\gamma_k} S_{d-1} \frac{ I^t(r)}{r}  \lesssim \frac{1}{t},
\end{equation}
with the implicit constant being uniform in $t>0,$ $ 0\le r \leqslant t,$ and the dimension $d.$ Note that for  $s\ge (\frac{t^2}{r^2}-1)^{1/2}$ we have $\frac{1}{r} \leqslant \frac{\sqrt{s^2+1}}{t}.$ Hence, recalling \eqref{eq:Ir} and using \eqref{eq:Bform} we obtain
\begin{align*}
\frac{\widetilde{\gamma_k}}{\gamma_k} S_{d-1}	\frac{I^t(r)}{r} &\leqslant \frac{\widetilde{\gamma_k}}{\gamma_k} S_{d-1} \frac{1}{t}\int_{\sqrt{\frac{t^2}{r^2}-1}}^\infty \frac{s^{d-1}}{\left( 1+s^2 \right)^{d+(k-1)/2}} \, ds \\
	&\leqslant S_{d-1}\frac{\widetilde{\gamma_k}}{\gamma_k} \frac{1}{t}\int_{0}^\infty \frac{s^{d-1}}{\left( 1+s^2 \right)^{d+(k-1)/2}} \, ds = S_{d-1}\frac{\widetilde{\gamma_k}}{\gamma_k}\frac{\Gamma\left( \frac{d+k-1}{2} \right) \Gamma\left( \frac{d}{2} \right)}{2\Gamma\left( d+\frac{k-1}{2} \right)}\cdot \frac{1}{t}.
\end{align*}

Applying \eqref{eq:Sd-1} and \eqref{eq:gk} we reach
\begin{align*}
S_{d-1}\frac{\widetilde{\gamma_k}}{\gamma_k} \frac{I^t(r)}{r}&\leq \frac{2\pi^{d/2}}{\Gamma\left( \frac{d}{2} \right)}\frac{\Gamma(d+\frac{k}{2})}{\pi^{d/2}\Gamma\left( \frac{d+k}{2} \right)}\frac{\Gamma\left( \frac{d+k-1}{2} \right) \Gamma\left( \frac{d}{2} \right)}{2\Gamma\left( d+\frac{k-1}{2} \right)}\cdot \frac{1}{t}\\
&=\frac{\Gamma(d+\frac{k}{2})}{\Gamma\left( d+\frac{k-1}{2} \right)}\cdot \frac{\Gamma\left( \frac{d+k-1}{2} \right)}{\Gamma\left( \frac{d+k}{2} \right)}\cdot \frac{1}{t}.
\end{align*}
Since $k$ is fixed, using \eqref{StirFra} we conclude that
 \begin{align*}
S_{d-1}\frac{\widetilde{\gamma_k}}{\gamma_k} 	\frac{I^t(r)}{r}\lesssim \frac{\left(d+\frac{k-1}{2}\right)^{1/2}}{\left(\frac{d}{2}+\frac{k-1}{2}\right)^{1/2}} \cdot \frac{1}{t}\lesssim \frac{1}{t}.
 \end{align*}
Thus, we completed the proof of \eqref{eq:mHmM} and hence also the proof of \thref{pro:mhzt}.

\end{proof}

We will also need vector-valued estimates for $\{R_j(f_s)\},$ $j\in I,$ $s=1,\ldots,d.$ The following proposition can be deduced from \thref{pro:tRve} if we proceed along the lines of \cite[Section 4]{iwaniec_martin}.
\begin{pro}
	\thlabel{pro:Rve}
	For each $1<p<\infty$ we have 
	\begin{equation}
		\label{eq:Rve}
		\norm{\left(\sum_{s=1}^S \sum_{j\in I} |R_{j}f_s|^2\right)^{1/2}}_{L^p(\R^d)} \lesssim p^* p^{1/2} q^{\frac{k+1}{2}} \norm{ \left(\sum_{s=1}^S | f_s|^2\right)^{1/2}}_{L^p(\R^d)},
	\end{equation}	
	\begin{equation}
		\label{eq:Rvs}
		\norm{\left( \sum_{j\in I} |R_{j}f|^2\right)^{1/2}}_{L^p(\R^d)} \lesssim p^* q^{k/2} \norm{f}_{L^p(\R^d)},
	\end{equation}	
	uniformly in the dimension $d.$
\end{pro} 

\begin{proof}
	In contrast to the proofs of \thref{thm1tres} and \thref{thm2tres} here we apply the methods from \cite[Section 4]{iwaniec_martin} in a direct way. Therefore we shall be brief. Let $n=k=d$ and identify $\CC^d$ with $\R^{2d}.$  
	
	For the proof \eqref{eq:Rve} we take $E=\ell^2(\{1,\ldots,S\})$ and $F=\ell^2(\{1,\ldots,S\}\times I).$ The operator ${\bf T}$  is defined by
	\[
		{\bf T}(\{f_s\}_{s=1,\ldots,S})= \{\tR_j(f_s)\}_{(s,j)\in \{1,\ldots,S\}\times I}.
	\]
	Using \eqref{eq:m} for $P(z)=z_{j_1}\cdots z_{j_k}$ one can check that the restricted operator ${\bf T}_0$ is then  
	\[
		{\bf T}_0(\{f_s\}_{s=1,\ldots,S})= \{R_j(f_s)\}_{(s,j)\in \{1,\ldots,S\}\times I}.
	\]
	Hence, \cite[eq.\ (45)]{iwaniec_martin} together with \eqref{eq:tRve} lead to \eqref{eq:Rve}.
	
	The proof of \eqref{eq:Rvs} is similar. We take $E=\CC$ and  $F=\ell^2(I).$ The operators ${\bf T}$ and $\bf{T}_0$ are defined as above. The desired inequality follows from \cite[eq.\ (45)]{iwaniec_martin} together with \eqref{eq:tRvs}.
\end{proof}

We are finally ready to justify \thref{thm1quat,thm2quat}. At this point the proofs mimic the corresponding proofs of  \thref{thm1c'',thm2c''}. Therefore we shall be brief and only point out the differences.

\begin{proof}[Proof of \thref{thm1quat}]
	We proceed analogously to the proof of \thref{thm1c''} on p.\ \pageref{p:pfthm1c''}. In particular, we replace $\CC^d$ with $\R^d,$ $\tR_{j}^{t_n}$ with $D_{j}^{t_n}$ and $\tR_j$ with $R_j.$ The most important difference is that \eqref{eq:Djtrot} replaces \eqref{eq:rot2}. This leads to the replacement of \eqref{eq:tRtfHz} by
	\begin{equation}
		\label{eq:DtfHx}
	D^t f(x)= -\frac{2\Gamma\left( \frac{k+d}{2}\right)}{\Gamma\left( \frac{k}{2}\right)\Gamma\left( \frac{d}{2}\right)}\int_{S^{d-1}}  \mH^t_{\omega}\left[\sum_{j\in I} \omega_j R_j f \right](x)\, d\omega.
	\end{equation} 
In the proof we also use \eqref{eq:Rve} in place of \eqref{eq:tRve} and  \thref{pro:mhzt} instead of \thref{pro:hzt}. 

\end{proof}

\begin{proof}[Proof of \thref{thm2quat}]
	We proceed analogously to the proof of \thref{thm2c''} on p.\ \pageref{p:pfthm2c''} making the replacements as in the proof of \thref{thm1quat}. In particular we use \eqref{eq:DtfHx},  \eqref{eq:Rvs}, and \thref{pro:mhzt}.
\end{proof}

\section{Appendix}
\label{sec:app}

\begin{proof}[Proof of \thref{pro:hzt}]
	A (complex) rotational invariance argument reduces the inequality to its one-dimensional case
	\begin{equation*}
		\norm{\left(\sum_{s=1}^S \abs{H^* f_s}^2\right)^{1/2}}_{L^p(\CC)} \lesssim p^* \norm{ 	\left(\sum_{s=1}^S \abs{f_s}^2\right)^{1/2}}_{L^p(\CC)}.
	\end{equation*}
	Here 
	\[
		H^*f(z):=\sup_{t\in \mathbb{Q}_+} |H_{k}^tf(z)|,\qquad \textrm{with}\quad 	H_{k}^tf(z)=\int_\CC \left( \frac{\lambda}{\abs{\lambda}} \right)^k \frac{f(z-\lambda )}{\abs{\lambda}^2} \ind{|\la|\ge t}(\lambda) \, d\lambda
	\]
	is the $k$-th power of the complex Hilbert transform on $\CC.$

	We split the operator $H^*$ into two parts. To this end let $\varphi: \CC \to \R$ be a smooth radial function satisfying $\varphi(z) = 1$ for $\abs{z} < 2$, $\varphi(z) = 0$ for $\abs{z} > 4$. Define $\varphi_t(z) = \varphi(z/t)$ and let
	\[
		\chi_t(z) = \left( \frac{z}{\abs{z}} \right)^k \frac{1}{\abs{z}^2} \ind{\abs{z} \geqslant t}
	\]
	be the kernel of $H^t_k$.
	Then
	\begin{align*}
		H^*f(z) &\leqslant \sup_{t>0} \abs{(\varphi_t\chi_t * f)(z)} + \sup_{t>0} 	\abs{((1-\varphi_t)\chi_t * f)(z)} \\
		&\eqqcolon H_\varphi^* f(z) + H_{1-\varphi}^* f(z) \\
		&\lesssim \mathcal{M}f(z) + H_{1-\varphi}^* f(z),
	\end{align*}
	where $\mathcal{M}$ denotes the Hardy--Littlewood maximal operator on $\R^2$. Since \cite[Theorem 5.6.6]{grafakos} gives us vector-valued estimates for $\mathcal M$ we get
	\[
		\norm{\left(\sum_{s=1}^S \abs{H_{\varphi}^* f_s}^2\right)^{1/2}}_{L^p(\CC)} \lesssim p^* 	\norm{ \left(\sum_{s=1}^S \abs{f_s}^2\right)^{1/2}}_{L^p(\CC)}.
	\]
	The remaining ingredient is to prove
	\begin{equation} \label{eq:H1phi}
		\norm{\left(\sum_{s=1}^S \abs{H_{1-\varphi}^* f_s}^2\right)^{1/2}}_{L^p(\CC)} \lesssim p^* 	\norm{ \left(\sum_{s=1}^S \abs{f_s}^2\right)^{1/2}}_{L^p(\CC)}.
	\end{equation}
	We will apply \cite[Theorem 5.6.1]{grafakos} with
	\[
		\mathcal{B}_1 = \ell^2\left(\{1, \dots, S\} \right) \qquad \textrm{and} \qquad \mathcal{B}_2 = \ell^2\left(\{1, \dots, S\}; L^\infty{(\mathbb{Q}_+)} \right)
	\]
	and
	\begin{equation} \label{eq:kernelK}
		\vec{K}(z)(u) = \left( (1-\varphi_{t}) \chi_{t}(z) \cdot u_1 ,\ldots, (1-\varphi_{t}) 	\chi_{t}(z) \cdot u_S\right) \in \mathcal{B}_2
	\end{equation}
	for any sequence $u=(u_s)_{s=1}^S \in \mathcal{B}_1$. Then, taking $e_s = (0, \dots, 1, \dots, 0)$, with $1$ on the $s$-th coordinate, we see that the operator $\vec{T}$ defined in \cite[5.6.4]{grafakos} satisfies
	\begin{equation} \label{eq:def_T}
		\vec{T}\left( \sum_{s=1}^S f_s e_s \right)(z) = \left(H^{t}_{1-\varphi} 	f_1(z),\ldots,H^{t}_{1-\varphi} f_S(z)\right)
	\end{equation}
	and
	\[
		\norm{\vec{T}\left( \sum_{s=1}^S f_s  e_s \right)(z)}_{{\mathcal{B}_2}} = \left(\sum_{s=1}^S \abs{H^*_{1-\varphi} f_s(z)}^2 \right)^{1/2}
	\]
	for any sequence $(f_s)_{s=1}^S$ of smooth functions that vanish at infinity.
	In order to use \cite[Theorem 5.6.1]{grafakos} we need to verify conditions (5.6.1), (5.6.2) and (5.6.3) from \cite{grafakos} and check that $\vec{T}$ is bounded from $L^2(\CC, \mathcal{B}_1)$ to $L^2(\CC, \mathcal{B}_2)$.
	
	Condition (5.6.1) is a straightforward consequence of \eqref{eq:kernelK}. It is also not hard to verify that $\int_{\varepsilon \le |z|\le 1}	\vec{K}(z)\,dz=0,$ so that condition (5.6.3) is satisfied with $\vec{K}_0=0$.

	We shall now justify (5.6.2). Denote $\widetilde{\varphi}_t \coloneqq 1-\varphi_t$ and $g_t = \widetilde{\varphi}_t \chi_t$ so that
		\[
	g_t(z) = \widetilde{\varphi}_t(z) \frac{z^k}{\abs{z}^{k+2}}.
	\]
Since
	\[
		\norm{\vec{K}(z-w) - \vec{K}(z)}_{\mathcal{B}_1 \to \mathcal{B}_2} = \sup_{t>0} 	\abs{g_t(z-w) - g_t(z)},
	\]
	we have
	\begin{align} \label{eq:supphi}
		&\norm{\vec{K}(z-w) - \vec{K}(z)}_{\mathcal{B}_1 \to \mathcal{B}_2} = \sup_{t>0} \abs{\widetilde{\varphi}_t(z-w) 	\frac{(z-w)^k}{\abs{z-w}^{k+2}} - \widetilde{\varphi}_t(z) \frac{z^k}{\abs{z}^{k+2}}} \nonumber \\
		&\leqslant \sup_{t>0} \abs{\left(\widetilde{\varphi}_t(z-w) - \widetilde{\varphi}_t(z) 	\right) \frac{(z-w)^k}{\abs{z-w}^{k+2}}} + \sup_{t>0} \abs{ \widetilde{\varphi}_t(z) \left( \frac{(z-w)^k}{\abs{z-w}^{k+2}} - \frac{z^k}{\abs{z}^{k+2}} \right)}.
	\end{align}

Hence, the proof of (5.6.2) boils down to estimating the two terms in \eqref{eq:supphi} under the assumption $|z|\ge 2|w|.$  We begin with the first term. Since $|z|\ge 2|w|$ we have $|z|\approx \abs{z-w}.$ Hence, in order for the expression inside the absolute value to be nonzero, $t$ has to be comparable to $\abs{z}$ and $|z-w|.$  In that case, using the smoothness of $\varphi$ we obtain
	\begin{align*}
		\abs{\left(\widetilde{\varphi}_t(z-w) - \widetilde{\varphi}_t(z) \right) \frac{(z-w)^k}{\abs{z-w}^{k+2}}} &\lesssim \frac{\abs{w}}{2t}\frac{1}{\abs{z-w}^{2}} \approx \frac{\abs{w}}{\abs{z}\abs{z-w}^{2}} \approx \frac{\abs{w}}{\abs{z}^3}.
	\end{align*}
	In the second term of \eqref{eq:supphi} we omit $\widetilde{\varphi}_t$ and get
	\begin{align*}
		&\abs{\frac{(z-w)^k}{\abs{z-w}^{k+2}} - \frac{z^k}{\abs{z}^{k+2}}} \leqslant \abs{\frac{(z-w)^k}{\abs{z-w}^{k+2}} - \frac{(z-w)^k}{\abs{z}^{k+2}}} + \abs{\frac{(z-w)^k}{\abs{z}^{k+2}} -  \frac{z^k}{\abs{z}^{k+2}}} \\
		&= \abs{z-w}^k \frac{\abs{\abs{z}^{k+2} - \abs{z-w}^{k+2}}}{\abs{z-w}^{k+2}\abs{z}^{k+2}} + \frac{1}{\abs{z}^{k+2}} \abs{(z-w)^k - z^k} \approx \frac{\abs{w}}{\abs{z}^3}.
	\end{align*}
	This means that we have proved that
	\[
		\norm{\vec{K}(z-w) - \vec{K}(z)}_{\mathcal{B}_1 \to \mathcal{B}_2} \lesssim \frac{\abs{w}}{\abs{z}^3}
	\]
	for $\abs{z} \geqslant 2\abs{w}$. Integrating this yields
	\[
		\int_{\abs{z} \geqslant 2\abs{w}} \norm{\vec{K}(z-w) - \vec{K}(z)}_{\mathcal{B}_1 \to 	\mathcal{B}_2} dz \lesssim \abs{w} \int_{\abs{z} \geqslant 2\abs{w}} \frac{1}{\abs{z}^3} dz \approx 1
	\]
	so that condition (5.6.2) is satisfied. 
	
	It remains to justify the boundedness of $\vec{T}$ from $L^2(\CC, \mathcal{B}_1)$ to $L^2(\CC, \mathcal{B}_2).$ We have the pointwise bound
	\[
		H^*_{1-\varphi}f(z) \lesssim \mathcal{M}f(z) + H^* f(z).
	\]
	Therefore the desired $L^2$ boundedness of $\vec{T}$ is  a consequence of \eqref{eq:def_T} and the $L^2(\CC)$ boundedness of $H^*$. This allows us to use \cite[Theorem 5.6.1]{grafakos} and completes the proof of \eqref{eq:H1phi} hence also the proof of \thref{pro:hzt}.
\end{proof}

\end{document}